\documentclass[masc,preprint]{informs3}

\usepackage[boxruled,linesnumbered]{algorithm2e}
\usepackage[caption=false]{subfig}

\newtheorem{remark}{Remark}
\newtheorem{definition}{Definition}

\newtheorem{example}{Example}

\newtheorem{proposition}{Proposition}

\newtheorem{corol}{Corollary}

\newcounter{Propcount}
\usepackage[sort&compress,longnamesfirst]{natbib}
\bibpunct[, ]{(}{)}{,}{a}{}{,}%
\def\bibfont{\small}%

\usepackage{bbm}

\usepackage{pgfplots}
\pgfplotsset{compat=1.18}

\renewcommand{\theARTICLETOP}{} 

\makeatletter
\def\setoddRH{\hbox to \textwidth{\hfill\fs.10.10.\thepage}}
\def\setevenRH{\hbox to \textwidth{\fs.10.10.\thepage\hfill}}
\makeatother

\begin{document}

\RUNAUTHOR{Boomsma, Devine and Anjos}
\RUNTITLE{Unit commitment constrained Nash equilibrium in power markets}
\TITLE{Unit commitment constrained Nash equilibrium in power markets}

\ARTICLEAUTHORS{%
\AUTHOR{Trine Krogh Boomsma$^*$}
\AFF{Department of Mathematical Sciences, University of Copenhagen, 2100 K\o benhavn \O, Denmark}
\AFF{$^*$ Corresponding author, \EMAIL{trine@math.ku.dk}}
\AUTHOR{Mel T. Devine}
\AFF{University College Dublin, Energy Institute, College of Business, Dublin 4}
\AUTHOR{Miguel F. Anjos}
\AFF{University of Edinburgh, School of Mathematics, James Clerk Maxwell Building, Peter Guthrie Tait Road, Edinburgh, EH9 3FD}
} 

\ABSTRACT{Equilibrium modeling for power markets usually assumes convexity of each players' optimization problem. Although the importance of accounting for fixed generation costs and start-up costs, minimum generation levels and/or minimum up-time and down-time restrictions in production scheduling is widely acknowledged, such modeling does not allow for discrete decisions. This paper considers unit commitment constrained Nash equilibria. First, we derive novel optimality conditions tailored for the mixed-integer convex programming problem of joint unit commitment and economic dispatch of a self-scheduling profit-maximizing producer. Next, we use these to formulate a Nash equilibrium as a mixed-integer complementarity problem, which facilitates the development of a procedure to obtain multiple equilibria. The approach can be adapted to both perfectly and imperfectly competitive market settings. While an equilibrium may not always exist, we give sufficient conditions under which a Cournot-Nash equilibrium can be obtained by mixed-integer convex programming under the standard assumption of an affine inverse demand curve and convex generating costs, and use this to establish existence. A case study demonstrates the feasibility of using optimisation to obtain a Cournot-Nash equilibrium for unit commitment constrained market clearing. Our results confirm that ignoring unit commitment creates significant welfare losses, although these vary substantially across equilibria.}


\maketitle

\section{Introduction}

Since the liberalisation in the 1990s, modern electricity markets have been organised as auctions into which producers submit profit-maximising bids and generators are dispatched by a market operator. Analysing market dynamics and supporting market-clearing decisions therefore calls for a game-theoretical approach. Equilibrium modeling of competitive markets usually assumes convexity of each players’ optimisation problem. However, as already noted by \cite{Scarf1994}, most commodity markets in today's advanced economies involve considerable non-convexities in the form of discrete decisions. For power markets in particular, the discreteness of decisions is due to fixed operating costs and start-up costs, minimum generation levels and/or minimum up-time and down-time restrictions in generator scheduling. The importance of so-called unit commitment decisions is widely acknowledged in system-wide cost minimisation or social welfare maximisation. Start-up costs in particular have been shown to have a critical impact on the optimal dispatch, especially for power systems with high shares of variable renewable production (\cite{Aonghus2013}). Also, econometric analysis of historical bids has documented substantial start-up costs and argued that accounting for them would better explain observed bidding strategies and production patterns (\cite{Reguant2014}). Reflecting this, most electricity markets in the US allow for explicitly including fixed operating costs and start-up costs of the generators in the bids, whereas EU markets generally require producers to internalize these costs (Herrero (2020)). Naturally, non-convexity attracts increasing attention in modeling competitive electricity markets.  

In this paper, we model a power market in which producers have both continuous and discrete operating decisions as a discretely-constrained Nash equilibrium. Accordingly, we consider a number of self-scheduling producers, each operating one or more generating units, for which commitment involves binary decisions whereas dispatch decisions are continuous. We refer to our model as a unit commitment constrained Nash equilibrium. On the supply side, our formulation allows for discontinuous costs and non-connected operating regions of individual producers, whereas we represent the demand side by an aggregate inverse demand curve. Part of our analysis applies to both perfectly and imperfectly competitive markets, although our main focus is obtaining Cournot-Nash equilibria. 

Our contributions are the following. First, we derive tailored optimality conditions for the mixed-integer convex programming formulation of joint unit commitment and economic dispatch for a self-scheduling profit-maximising producer, by leveraging the binary nature of unit commitment. The optimality conditions include an exponential number of inequalities and complementarity constraints, but we demonstrate how this number can be significantly reduced in specific cases. For the profit-maximisation problem, we also derive valid inequalities that prove valuable for the subsequent equilibrium analysis. Second, we employ the optimality conditions to characterise a unit commitment constrained Nash equilibrium by an equation system and show how the valid inequalities reduce the total number of constraints required. This characterization of an equilibrium facilitates the development of a novel procedure for identifying multiple distinct equilibria. The approach can be adapted to both perfectly and imperfectly competitive market settings. Third, for the Cournot-Nash framework, we show that an equilibrium can be obtained by solving a single system-wide mixed-integer convex optimisation problem, under standard assumptions of an affine inverse demand curve and convex generating costs. While prior work has shown that an equilibrium may not exist under perfect competition, we use this convex optimisation problem to establish existence under Cournot competition. To the best of our knowledge, we are the first to analyse the unit commitment constrained Cournot-Nash equilibrium. Our theoretical results and notably our mixed-integer convex optimisation formulation offer market operators and policy-makers practical decision support, without the need for specialised algorithms and using standard optimisation software.

We carry out computational experiments using real-world data. Our results compare market equilibria with and without discreteness to demonstrate the importance of accounting for unit commitment decisions. Also, they compare the system-wide optimal equilibrium to alternative equilibria with respect to market prices and volumes. We use a case study of the all-island Irish electricity market, with a dataset including both technical and economic features of 52 generating units spread across 16 producers.

The paper is structured as follows. Section 2 provides an overview of related literature. In Sections 3 and 4, we consider the unit commitment problem of a single producer and the unit commitment constrained equilibrium problem, respectively. Section 5 discusses the numerical experiments and Section 6 concludes the paper.

\section{Related literature}

Unit commitment in power system optimisation has a long history. Originally, models would rely on central decision-making and system-wide social welfare maximisation and did not account for market characteristics such as competition. The liberalisation of electricity markets motivated an equilibrium perspective with representations of each player’s surplus maximisation problem. Convexity assumptions facilitate necessary and sufficient optimality conditions, the collection of which characterises the equilibrium and typically forms a mixed complementarity problem (MCP). Examples of MCPs for modelling electricity markets include \cite{Hobbs2001, Hobbs2007, Gabriel2013, Wogrin2013}, in connection to capacity expansion, \cite{Garcia2010, Ehrenmann2011}, and in a more general energy modeling context, \cite{Bushnell2003, Gabriel2005, Chen2020}. For a comprehensive overview, we refer the reader to \cite{Egging2020}. None of the cited references incorporate discrete decision-making such as unit commitment. 

Given the importance of discreteness, the modeling of electricity markets with non-convexities has attracted considerable attention in the mathematical optimisation literature (\cite{Philpott2006, Bjorndal2008, Liu2013, ONeill2013}). Under convexity and linearity, it is possible to prove existence of market clearing prices. Furthermore, there exists an equivalent system-wide optimisation problem, in which linear equilibrium prices are simply the dual variables of the market clearing constraints. In the presence of unit commitment, it is less clear how to define electricity prices. \cite{Humbs2022} analyse the special case in which the linear relaxation of each market player's optimisation problem solves their mixed-integer linear problem (MILP). By exploiting total unimodularity, they derive necessary and sufficient conditions for the existence of a price and demonstrate for an example with minimum up and down time constraints that produce the convex hull of the MILP. For strong formulations of the unit commitment problem, see also \cite{Ostrowski2012} and \cite{Pan2016} for the deterministic and stochastic problem, respectively. Generally, however, linear prices may not exist, if the linear relaxation of the market clearing problem fails to solve the mixed-integer problem. A common approach to determining prices is to minimize total lost opportunity costs of the market players under the system-wide optimal solution. The minimisation of global opportunity costs is known as convex hull pricing, and is obtained by convexifying all players' optimisation problems, see \cite{Gribik2007}. Minimisation of local opportunity costs can be achieved by solving the system-wide MILP, fixing the unit commitment and reoptimising the linear problem, with the resulting dual prices referred to as IP-prices. \cite{ONeill2005} determine IP-prices by adding equality constraints to the linear problem that force the original integer variables to assume their optimal values and obtain duals to both market clearing conditions and added constraints. Neither convex hull prices or IP-prices, however, are guaranteed to support an equilibrium. To ensure appropriate incentives, it may be necessary to introduce compensation to market players in the form of uplift or make-whole payments, offsetting lost opportunity costs or negative profits. The recent study by \cite{Ahunbay2024} compares such pricing schemes and compensation mechanisms and demonstrates numerically that focusing on one type of opportunity costs can cause substantial increases in other types of opportunity costs. To balance performance, the authors therefore suggest a multi-objective approach. In contrast to obtaining prices through opportunity cost minimisation, we aim at a market clearing price that internalises the costs of unit commitment and respects the incentives of all market players. We do not restrict ourselves to prices being dual variables, but rather define the price as a free variable that allows for balance between demand and supply, consistent with \cite{Gabriel2013, Todd2016, Humbs2022, Huppmann2018, Weinhold2020}.

In addition to defining and determining prices, it remains a challenge to obtain discretely constrained Nash equilibria. For the special case of a unit commitment constrained Nash equilibrium, \cite{Gabriel2013} suggest to relax the integrality of the unit commitment variables, formulate an MCP and reinforce integrality to arrive at a mixed-integer MCP. To avoid feasibility issues, complementarity is relaxed and violations are penalised in an objective function. As the solutions may not constitute a Nash equilibrium, \cite{Weinhold2020} extend the heuristic by an ex post check for incentive compatibility. An alternative approach is proposed by \cite{Devine2025} and involves a Gauss-Seidel type of algorithm, based on iteratively solving each producer’s problem while taking the decisions of other producers as given. Results for the Irish power system shows that neglecting unit commitment exaggerates the impact of price-making behaviour, documenting the need for simultaneously considering discreteness and imperfect competition. While the Gauss-Seidel algorithm can be shown to converge for Cournot competition under some assumptions, convergence of the algorithm is not guaranteed in the presence of unit commitment. Furthermore, the procedure involves repeatedly solving several optimisation problems. In contrast, the present work solves a single optimisation problem and only once to obtain a Cournot-Nash equilibrium.

A few references consider the general class of discretely constrained Nash equilibria. \cite{Todd2016} studies pure integer (unconstrained) Cournot equilibria and establishes existence of an equilibrium via a potential function approach, see also \cite{Monderer1996} for the continuous case. We impose the same assumptions of affine inverse demand and convex quadratic costs, but examine the mixed-integer problem of unit commitment and allow for various operational constraints in the producers' problems. Nevertheless, our approach to proving existence is similar in the sense that we establish an optimisation problem (although constrained), for which an optimal solution is an equilibrium. \cite{Todd2016} computes equilibria by solving continuous equilibrium problems and reinforcing integrality by binary search. For the even more general class of non-convex Nash equilibria, \cite{Fuller2017} and \cite{Harwood2024} suggest to obtain a minimum disequilibrium by optimisation, e.g.\ of total opportunity costs. The former solves the special case of mixed-integer convex programming by relaxation of integrality and complementarity, whereas the latter proposes a lower bounding procedure based on the relaxation of each player's optimality. Both references illustrate their method for a unit commitment constrained perfectly competitive Nash equilibrium, for which an equilibrium may not exist. For matrix games, \cite{Avis2010} and \cite{Audet2006} simply solve an equilibrium for all possible strategies and check for deviation incentives. This approach can be applied to the unit commitment constrained Nash equilibrium, but requires $2^{TI}$ continuous equilibrium problems to be solved, where $T$ is the number of time periods and $I$ is the number of generating units in the market. Our approach is exact like this method, but scales linearly with $I$.

Closest to our work is the study of binary Nash equilibria by \cite{Huppmann2018}. The idea behind our optimality conditions is similar in the sense that we also condition on the values of the binary variables and derive optimality conditions for the resulting continuous optimisation problem. Whereas the authors enforce optimality of the binaries by a reformulation of incentive compatibility using big-M constraints, we directly apply the inequalities defining a Nash equilibrium. As \cite{Huppmann2018}, we avoid the comparison of all binary solutions of all players by adding valid inequalities, but specify tight inequalities for the multidimensional binary problem of each player. Our focus on a tight formulation reduces the search space of equilibria and facilitates efficient computations. The most important difference from our work is their definition of an equilibrium price that balances social welfare and compensation payments to producers. We do not allow for compensation but simply determine a price that facilitates equilibrium (when it exists). Also, their analysis does not account for market power, whereas we cover Cournot-Nash equilibria. In particular, we show that the price that maximises market power adjusted social welfare supports an equilibrium under Cournot competition. 

\section{Optimality conditions for unit commitment and economic dispatch}

Consider the unit commitment (UC) problem of a profit-maximizing power producer. Let $u=(u_1^\top,\dots,u_I^\top)^\top\in\mathbb{R}^{IT}$ and $q=(q_1^\top,\dots,q_T^\top)^\top\in\mathbb{R}^{IT}$ represent binary unit commitment decisions and continuous dispatch decisions, respectively, of the generating units $i=1,\dots,I$ at the discrete time periods $t=1,\dots,T$ (for decisions of unit $i$, we use the notation $u_i=(u_{i1},\dots,u_{iT})^\top\in\mathbb{R}^{T}$, and for decisions at time $t$, we use $u_t=(u_{1t},\dots,u_{It})^\top\in\mathbb{R}^{I}$). The mixed-integer convex programming problem is:
\begin{align}
\max_{q,u} \Big\{&\sum_{t=1}^TP_t(q_{t})-\sum_{i=1}^I\sum_{t=1}^TC_i(q_{it})-\sum_{i=1}^IS_i(u_i) : \nonumber\\
&G_iq_i+H_iu_i\leq h_i, A_iu_i\leq b_i, u_i\in\{0,1\}^T, i=1,\dots,I\Big\},\label{eq1}
\end{align}
where $P_t:\mathbb{R}^{I}\rightarrow\mathbb{R}$ is a concave and continuously differentiable revenue function, $C_i:\mathbb{R}\rightarrow\mathbb{R}$ is a convex and continuously differentiable function describing variable operating costs, and $S_i:\mathbb{R}^{T}\rightarrow\mathbb{R}$ is a continuous function including fixed operating costs. Furthermore, the matrices and vectors $A_i\in\mathbb{R}^{n\times T}, b_i\in\mathbb{R}^{n}, G_i,H_i\in\mathbb{R}^{m\times T}$ and $h_i\in\mathbb{R}^{m}$ define operational constraints for $i=1,\dots,I$. Note that the constraints and costs are separable with respect to the production units, whereas this may not necessarily apply for revenue.

\begin{remark}
For example, revenue is given by a market price times total production $$P_t(q_1,\dots,q_I)=p_t\Big(\sum_{i=1}^Iq_{i}\Big)\sum_{i=1}^Iq_{i},$$ where the price is determined by an inverse affine demand curve $p_t(d_t)=\beta_t-\alpha_t d_t$, with $\beta_t,\alpha_t\geq 0$, and the balance between demand and supply $d_t=\sum_{i=1}^Iq_{i}$ for $t=1,\dots,T$. 

Moreover, variable costs are convex quadratic $$C_i(q)=a_iq+b_iq^2,$$ with $a_i,b_i\geq 0$, and fixed costs include online, start-up and shut-down costs $$S_i(u_1,\dots,u_T)=\sum_{t=1}^T(s_i^{ON}u_t+s_i^{SU}\max\{u_{t}-u_{t-1},0\}+s_i^{SD}\max\{u_{t-1}-u_{t},0\}),$$ with $s_i^{ON}, s_i^{SU},s_i^{SD}\geq 0$, for $i=1,\dots,I$.
\end{remark}

We assume that the problem (\ref{eq1}) is feasible and bounded such that it has an optimal solution. By exploiting the binary nature of unit commitment, the problem is equivalent to
\begin{align}
\max_{u} \Big\{&\max_{q}\Big\{\sum_{t=1}^TP_t(q_{t})-\sum_{i=1}^I\sum_{t=1}^TC_i(q_{it}): G_iq_i\leq h_i-H_iu_i, i=1,\dots,I\Big\}-\sum_{i=1}^IS_i(u_i) : \nonumber\\
&A_iu_i\leq b_i, u_i\in\{0,1\}^T, i=1,\dots,I\Big\}.\nonumber
\end{align}
For fixed $u$, the inner optimization is an economic dispatch (ED) problem, and we refer to it as $ED(u)$. This is a convex programming problem with a concave objective function and affine constraint functions. The Karush-Kuhn Tucker (KKT) conditions are both necessary and sufficient for optimality. Hence, $q$ is an optimal solution to $ED(u)$ if and only if there exists multipliers $\lambda_i\in\mathbb{R}^{m}, i=1,\dots,I$ such that the following conditions are satisfied, including stationarity, primal and dual feasibility and complementarity:
\begin{align}
-\frac{\partial P_t}{\partial q_{it}}(q_t)+\frac{dC_i}{dq_{it}}(q_{it})+\lambda_i^\top g_{it}=0, \ t=1,\dots,T, \ 0\leq \lambda_i\perp h_i-H_iu_i-G_iq_i\geq 0, \ i=1,\dots,I,\label{eq4}
\end{align}
where $g_{it}$ is the $t$-th column of $G_i$ and $x\perp y$ means $x^\top y=0$. This system of constraints forms a mixed complementarity problem (MCP).

\begin{remark}
The complementarity constraints
\begin{align*}
0\leq \lambda_i\perp h_i-H_iu_i-G_iq_i\geq 0,
\end{align*}
can be enforced by the mixed-integer linear programming formulation
\begin{align*}
0\leq \lambda_{i}\leq M_i\delta_i, 0\leq h_i-H_iu_i-G_iq_i\leq M_i(\mathbbm{1}-\delta_i), \ \delta_i\in\{0,1\}^m,
\end{align*}
where $\mathbbm{1}=(1,\dots,1)^\top\in\mathbb{R}^m$, provided there exist a sufficiently large constant $M_i$, see \cite{Fortuny-Amat1981}. 

For instance, if the constraints $G_iq_i+H_iu_i\leq h_i$ simply bound the production of a committed unit, i.e., $\underline q_iu_{it} \leq q_{it}\leq \overline q_iu_{it}, t=1,\dots,T$, where $\underline{q}_i$ and $\overline{q}_i$ are lower and upper bounds, respectively, and $P_t(q_1,\dots,q_I)$ $=(\beta_t-\alpha_t\sum_{i=1}^Iq_{it})\sum_{i=1}^Iq_{it}$ and $C_i(q)=a_iq+b_iq^2$, the KKT conditions are  
\begin{align*}
-\beta_t+2\alpha_t\sum_{k=1}^Iq_{kt}+a_i+2b_iq_{it}+\overline\lambda_{it}-\underline\lambda_{it}=0,\ 
0\leq \overline\lambda_{it}\perp \overline q_iu_{it}-q_{it}\geq 0, \ 
0\leq \underline\lambda_{it}\perp q_{it}-\underline q_iu_{it}\geq 0,\\
t=1,\dots,T, i=1,\dots, I,
\end{align*}
where $\underline\lambda_{it}$ and $\overline\lambda_{it}$ are the multipliers of the lower and upper bounding constraints. Noting that we can take $\overline\lambda_{it}\leq\beta_t-2\alpha_t\sum_{k=1}^Iq_{kt}-a_i-2b_iq_{it}$ and $ \underline\lambda_{it}\leq -\beta_t+2\alpha_t\sum_{k=1}^Iq_{kt}+a_i+2b_iq_{it}$, for instance, $M_i:=\max_t\{\max\{\beta_t-a_i,-\beta_t+2\alpha_t\sum_{k=1}^I\overline q_k+a_i+2b_i\overline q_i,\overline q_i\}\}$.
\end{remark}

\begin{remark}
The Karush-Kuhn Tucker conditions \eqref{eq4} can be replaced by alternative optimality conditions, e.g., primal and dual feasibility and strong duality:
\begin{align*}
\lambda_i, h_i-H_iu_i-G_iq_i\geq 0, \ i=1\dots,I, \ \sum_{t=1}^TP_t(q_t)-\sum_{i=1}^I\sum_{t=1}^TC_i(q_{it})=\sum_{i=1}^I\lambda_i^\top(h_i-H_iu_i).
\end{align*}
\end{remark}

Since $\{u\in\{0,1\}^{IT}: A_iu_i\leq b_i, i=1,\dots,I\}$ contains finitely many vectors, we can enumerate them $u^j,j=1,\dots,J$, where $J\leq 2^{IT}$. For fixed $j$, $q^j$ is an optimal solution to $ED(u^j)$ if and only if the KKT-conditions are satisfied. Evidently, a solution is optimal if and only if it is feasible and its objective function value is at least as high as the value of all solutions $(q^j,u^j), j=1,\dots,J$. 

\begin{proposition}\label{prop1} A solution $(q,u)$ is optimal to \eqref{eq1} if and only if 
\begin{align}
G_iq_i+H_iu_i\leq h_i, \ A_iu_i\leq b_i, \ u_i\in\{0,1\}, \ i=1,\dots,I,\label{eq5}\\
\sum_{t=1}^TP_t(q_{t})-\sum_{i=1}^I\sum_{t=1}^TC_i(q_{it})-\sum_{i=1}^IS_i(u_i)\geq \sum_{t=1}^TP_t(q_{t}^{j})-\sum_{i=1}^I\sum_{t=1}^TC_i(q_{it}^j)-\sum_{i=1}^IS_i(u_i^j), \ j=1,\dots,J,\label{eq2}\\
-\frac{\partial P_t}{\partial q_{it}}(q_t^j)+\frac{dC_i}{dq_{it}}(q_{it}^j)+(\lambda_{i}^j)^\top g_{it}=0, \ t=1,\dots,T, \ 0\leq \lambda_i^j\perp h_i-G_iq_i^j-H_iu_i^j\geq 0,\nonumber\\ 
i=1,\dots,I, j=1,\dots,J.\label{eq3}
\end{align}
\end{proposition}

This system of inequalities and complementarity constraints characterises an optimal solution and will be useful for the equilibrium problem. The complementarity constraints (\ref{eq3}) are clearly nonlinear, but can be enforced by mixed-integer linear programming. The inequalities (\ref{eq2}) may likewise be nonlinear, but are linear in the special case that $P_t(\cdot),C_i(\cdot)$ and $S_i(\cdot)$ are all linear. In the worst case, the number of the inequalities \eqref{eq2} and the number of complementarity constraints \eqref{eq3} are both exponential in $I$ and $T$. The number of complementarity constraints \eqref{eq3}, however, can often be much smaller, as illustrated by the following examples.

\begin{example}\label{ex1} 
If $({\partial P_t}/{\partial q_{it}})(q_{t}):=({\partial P_t}/{\partial q_{it}})(q_{it})$, e.g.\ if $P_t(q_1,\dots,q_I)=p_t\sum_{i=1}^Iq_{it}$ and the producer is a price-taker, then the KKT conditions are separate for each unit and \eqref{eq3} can be replaced by 
\begin{align*}
-\frac{\partial P_t}{\partial q_{it}}(q_{it}^l)+\frac{dC_i}{dq_{it}}(q_{it}^l)+(\lambda_{i}^l)^\top g_{it}=0, \ t=1,\dots,T, \ 0\leq \lambda_i^l\perp h_i-G_iq_i^l-H_iu_i^l\geq 0, \\
l=1,\dots,L_i, i=1,\dots,I,
\end{align*}
where $L_i\leq 2^T$ and $q_{i}^j=q_{i}^l$ if $u_{i}^j=u_{i}^l$  in \eqref{eq5} and \eqref{eq2}. As a result, the number of complementarity constraints increases only linearly with $I$. 
\end{example}

\begin{example}\label{ex2} 
If the constraints $G_iq_i+H_iu_i\leq h_i$ simply bound the production of a committed unit, i.e., $\underline q_iu_{it} \leq q_{it}\leq \overline q_iu_{it}, t=1,\dots,T$, the KKT conditions are separate for each time period. Then, \eqref{eq3} reduces to
\begin{align*}
-\frac{\partial P_t}{\partial q_{it}}(q_{t}^l)+\frac{dC_i}{dq_{it}}(q_{it}^l)+\overline{\lambda}_{it}^l-\underline{\lambda}_{it}^l=0, \ 0\leq \overline{\lambda}_{it}^l\perp \overline{q}_iu_{it}^l-q_{it}^l\geq 0, \ 0\leq \underline{\lambda}_{it}^l\perp q_{it}^l-\underline{q}_iu_{it}^l\geq 0,\\
l=1,\dots,L_t, t=1,\dots,T,i=1,\dots,I, 
\end{align*}
where $L_t\leq 2^I$ and $q_{t}^j=q_{t}^l$ if $u_{t}^j=u_{t}^l$ in \eqref{eq5} and \eqref{eq2}. The number of complementarity constraints increases only linearly with $T$.

If also $({\partial P_t}/{\partial q_{it}})(q_{t}):=({\partial P_t}/{\partial q_{it}})(q_{it})$, then \eqref{eq3} reduces to
\begin{align*}
q_{it}^0=0,\ -\frac{\partial P_t}{\partial q_{it}}(q_{it}^1)+\frac{dC_i}{dq_{it}}(q_{it}^1)+\overline{\lambda}_{it}-\underline{\lambda}_{it}=0, \ 0\leq \overline{\lambda}_{it}\perp \overline{q}_i-q_{it}^1\geq 0, \ 0\leq \underline{\lambda}_{it}\perp q_{it}^1-\underline{q}_i\geq 0,\\
t=1,\dots,T,  i=1,\dots,I,
\end{align*}
and $q_{it}^j=q_{it}^0$ if $u_{it}^j=0$ and $q_{it}^j=q_{it}^1$ if $u_{it}^j=1$  in \eqref{eq5} and \eqref{eq2}. 
\end{example}

\begin{example} 
With ramping restrictions, the constraints $G_iq_i+H_iu_i\leq h_i$ involve consecutive time periods:
\begin{align*}
u_{it}-u_{it-1}=v_{it}-w_{it}, \\
\underline{r}_iu_{it}-(\overline{r}_i-\overline{r}_i^{SU})v_{it}-(\underline{r}_i-\underline{r}_i^{SD})w_{it}\leq q_{it}-q_{it-1}
\leq \overline{r}_iu_{it}-(\overline{r}_i-\overline{r}_i^{SU})v_{it}-(\overline{r}_i-\overline{r}_i^{SD})w_{it},\\[1mm]
t=1,\dots,T,i=1,\dots,I,
\end{align*}
where $v_{it}\in\{0,1\}$ and $w_{it}\in\{0,1\}$ represent the start-up and shut-down decisions, respectively, $\overline{r}_i$ and $\underline{r}_i$ are ramping limits during production and $\overline{r}_i^{SU}, \underline{r}_i^{SU}, \overline{r}_i^{SD}, \underline{r}_i^{SD}$ are specific start-up and shut-down limits. If also $({\partial P_t}/{\partial q_{it}})(q_{t}):=({\partial P_t}/{\partial q_{it}})(q_{it})$, \eqref{eq3} can be replaced by
\begin{align*}
u_{it}-u_{it-1}=v_{it}-w_{it}, \\
q_{it_1}^{t_1,t_2}=q_{it_2}^{t_1,t_2}=0, \ -\frac{\partial P_t}{\partial q_{it}}(q_{it}^{t_1,t_2})+\frac{dC_i}{dq_{it}}(q_{it}^{t_1,t_2})+\overline{\lambda}_{it}-\underline{\lambda}_{it}-\overline{\gamma}_{it}+\underline{\gamma}_{it}=0,\\ 
0\leq \overline{\lambda}_{it}\perp  q_{it-1}^{t_1,t_2}-q_{it}^{t_1,t_2}+\overline{r}_iu_{it}-(\overline{r}_i-\overline{r}_i^{SU})v_{it}-(\overline{r}_i-\overline{r}_i^{SD})w_{it}\geq 0,\\[1mm]
0\leq \underline{\lambda}_{it}\perp q_{it}^{t_1,t_2}-q_{it-1}^{t_1,t_2}-\underline{r}_iu_{it}+(\overline{r}_i-\overline{r}_i^{SU})v_{it}+(\underline{r}_i-\underline{r}_i^{SD})w_{it}\geq 0,\\[1mm]
0\leq \overline{\gamma}_{it}\perp  q_{it}^{t_1,t_2}-q_{it+1}^{t_1,t_2}+\underline{r}_iu_{it+1}-(\overline{r}_i-\overline{r}_i^{SU})v_{it+1}-(\underline{r}_i-\underline{r}_i^{SD})w_{it+1}\geq 0, \\[1mm]
0\leq \underline{\gamma}_{it}\perp q_{it+1}^{t_1,t_2}-q_{it}^{t_1,t_2}-\underline{r}_iu_{it+1}+(\overline{r}_i-\overline{r}_i^{SU})v_{it+1}+(\underline{r}_i-\underline{r}_i^{SD})w_{it+1}\geq 0\\[1mm]
t=t_1+1,\dots,t_2-1,t_1=0,\dots,T+1,t_2=t_1+1,\dots,T, i=1,\dots,I,
\end{align*}
where $\underline{\gamma}_{it}$ and $\overline{\gamma}_{it}$ are multipliers of the ramping constraints, and $q_{it}^j=q_{it}^{t_1,t_2}, t=t_1,\dots,t_2$ if $u_{it}^j=u_{it}^{t_1,t_2}, t=t_1,\dots,t_2$. The number of complementarity constraints is cubic in $T$.
\end{example}

We proceed to derive valid inequalities that will be valuable for the characterisation of an equilibrium. Note that $\sum_{(i,t):u_{it}^j=1}(1-u_{it})+\sum_{(i,t):u_{it}^j=0}u_{it}=0$ for $u=u^j$ and $\sum_{(i,t):u_{it}^j=1}(1-u_{it})+\sum_{(i,t):u_{it}^j=0}u_{it}\geq 1$ for $u\neq u^j$. Thus, we can obtain valid inequalities for (\ref{eq5})-(\ref{eq3}) of the form:
\begin{align}
&q_{it}^j-\underline{M}_{ij}\Big(\sum_{(k,s):u_{ks}^j=1}(1-u_{ks})+\sum_{(k,s):u_{k,s}^j=0}u_{ks}\Big)\leq q_{it}\nonumber\\
&\leq q_{it}^j+\overline{M}_{ij}\Big(\sum_{(k,s):u_{ks}^j=1}(1-u_{ks})+\sum_{(k,s):u_{k,s}^j=0}u_{ks}\Big), \ t=1,\dots,T,i=1,\dots,I, j=1,\dots,J\label{eq10}
\end{align}
for sufficiently large constants $\underline{M}_{ij}$ and $\overline{M}_{ij}$. It should be remarked that these constraints ensure that $q=q^j$ if $u=u^j$.

We determine specific constants, assuming that $\overline{q}_iu_{it}\leq q_{it}\leq \overline{q}_iu_{it}$. If $u=u^j$ and $u_{it}^j=0$, then $u_{it}=0$ and $q_{it}^j=0=q_{it}$. Thus, we only impose (\ref{eq10}) for $u_{it}^j=1$. 

\begin{proposition}\label{prop5}
Assume that the constraints $G_iq_i+H_iu_i\leq h_i$ include $\overline{q}_iu_{it}\leq q_{it}\leq \overline{q}_iu_{it}$. The following are valid inequalities for (\ref{eq5})-(\ref{eq3}):
\begin{align*}
&q_{it}^j-(\overline{q}_i-\underline{q}_i)\Big(\sum_{(k,s)\neq (i,t):u_{ks}^j=1}(1-u_{ks})+\sum_{(k,s):u_{k,s}^j=0}u_{ks}\Big)-\overline{q}_i(1-u_{it})\leq q_{it}\nonumber\\
&\leq q_{it}^j+(\overline{q}_i-\underline{q}_i)\Big(\sum_{(k,s)\neq (i,t):u_{ks}^j=1}(1-u_{ks})+\sum_{(k,s):u_{k,s}^j=0}u_{ks}\Big)-\underline{q}_i(1-u_{it}).
\end{align*}
for $t=1,\dots,T,i=1,\dots,I, j=1,\dots,J$ with $u_{it}^j=1$.
\end{proposition}

The proof can be found in Appendix \ref{app:prop5}.

\begin{example}
Assume that the constraints $G_iq_i+H_iu_i\leq h_i$ only include $\overline{q}_iu_{it}\leq q_{it}\leq \overline{q}_iu_{it}$ and $({\partial P_t}/{\partial q_{it}})(q_{t}):=({\partial P_t}/{\partial q_{it}})(q_{it})$. Then, the following are stronger valid inequalities:
\begin{align*}
q_{it}^1-\overline{q}_i(1-u_{it})\leq q_{it}\leq q_{it}^1-\underline{q}_i(1-u_{it}), \ t=1,\dots,T, i=1,\dots,I.
\end{align*}
\end{example} 

\section{Nash equilibrium}

We proceed to demonstrate the use of the optimality conditions of (\ref{eq5})-(\ref{eq3}) to characterise the unit commitment constrained Nash equilibrium. We consider competing producers making simultaneous decisions in the market. For ease of exposition, each producer operates a single generating unit (we relax this assumption in the numerical experiments). The unit commitment constrained Nash equilibrium is defined as follows:

\begin{definition}
A unit commitment constrained Nash equilibrium is given by $(q_1,u_1,\dots,q_I,u_I)$ such that $(q_i,u_i)$ is an optimal solution to  
\begin{align}
\max_{q_i,u_i} \Big\{&\sum_{t=1}^TP_{it}(q_{it},q_{-it})-\sum_{t=1}^TC_i(q_{it})-S_i(u_i) :\nonumber\\
&G_iq_i+H_iu_i\leq h_i, A_iu_i\leq b_i, u_i\in\{0,1\}^T\Big\},\ i=1,\dots,I.\label{eq15}
\end{align}
given $q_{-i}=(q_1,\dots,q_{i-1},q_{i+1},\dots,q_T)$ for all $i=1,\dots,I$. Equivalently, $G_iq_i+H_iu_i\leq h_i, \ A_iu_i\leq b_i, \ u_i\in\{0,1\}^T, \ i=1,\dots,I$ and $(q_i,u_i)$ is the best response of unit $i$ given $q_{-i}$, i.e., 
\begin{align*}
\sum_{t=1}^TP_{it}(q_{it},q_{-it})-\sum_{t=1}^TC_i(q_{it})-S_i(u_i)\geq \sum_{t=1}^TP_{it}(\bar q_{it},q_{-it})-\sum_{t=1}^TC_i(\bar q_{it})-S_i(\bar u_i),\\
\forall (\bar q_i,\bar u_i)\in\{(q_i,u_i)\in\mathbb{R}^{T}\times\{0,1\}^T:G_iq_i+H_iu_i\leq h_i, A_iu_i\leq b_i\}, i=1,\dots,I.
\end{align*}
\end{definition}

By Proposition \ref{prop1}, this is equivalent to the system of constraints:
\begin{proposition}
$(q_1,u_1,\dots,q_I,u_I)$ is a unit commitment constrained Nash equilibrium if and only if
\begin{align}
G_iq_i+H_iu_i\leq h_i, \ A_iu_i\leq b_i, \ u_i\in\{0,1\}^T, \ i=1,\dots,I,\label{eq12}\\
\sum_{t=1}^TP_{it}(q_{it},q_{-it})-\sum_{t=1}^TC_i(q_{it})-S_i(u_i)\geq \sum_{t=1}^TP_{it}(q_{it}^l,q_{-it})-\sum_{t=1}^TC_i(q_{it}^l)-S_i(u_i^l),\nonumber\\ 
l=1,\dots,L_i,i=1,\dots,I,\label{eq17}\\
-\frac{\partial P_{it}}{\partial q_{it}}(q_{it}^l,q_{-it})+\frac{dC_i}{dq_{it}}(q_{it}^l)+(\lambda_{i}^l)^\top g_{it}=0, \ t=1,\dots,T, \ 0\leq \lambda_i^l\perp h_i-G_iq_i^l-H_iu_i^l\geq 0,\nonumber\\ 
l=1,\dots,L_i, i=1,\dots,I,\label{eq13}\\
q_{it}^l-\underline{M}_{il}\Big(\sum_{s:u_{is}^l=1}(1-u_{is})+\sum_{s:u_{is}^l=0}u_{is}\Big)\leq q_{it}\leq q_{it}^l+\overline{M}_{il}\Big(\sum_{s:u_{is}^l=1}(1-u_{is})+\sum_{s:u_{is}^l=0}u_{is}\Big), \nonumber\\ 
t=1,\dots,T,l=1,\dots,L_i,i=1,\dots,I.\label{eq27}
\end{align}
\end{proposition}
Note that we explicitly ensure that $q_i=q_i^l$ if $u_i=u_i^l$ by the inequalities (\ref{eq27}), avoiding the need to evaluate all possible combinations $(q_1^{l_1},u_1^{l_1},\dots,q_I^{l_I},u_I^{l_I}), l_1=1,\dots,L_1,\dots,l_I=1,\dots,L_I$. The total number of constraints in the equilibrium problem is proportional to $I2^T$, i.e, it increases linearly with $I$.

\begin{remark}
Consider the standard assumptions of an inverse affine demand curve, i.e., $P_{it}(q_t)=p_t(d_t)q_{it}$ with $p_t(d_t)=\beta_t-\alpha_t d_t, \alpha_t, \beta_t\geq 0$ and $d_t=\sum_{i=1}^Iq_{it}$. Let the unit commitment be fixed. Under Cournot competition, $({\partial p_{t}}/{\partial q_{it}})(d_t)=-\alpha_t$ such that $({\partial P_{it}}/{\partial q_{it}})(q_{it},q_{-it})=\beta_t-\alpha_t\sum_{k=1}^Iq_{kt}-\alpha_tq_{it}$. Under perfect competition, $({\partial p_{t}}/{\partial q_{it}})(d_t)=0$ such that $({\partial P_{it}}/{\partial q_{it}})(q_{it},q_{-it})=\beta_t-\alpha_t\sum_{k=1}^Iq_{kt}$. 
\end{remark}

\begin{remark} 
By the definition of Huppmann and Siddiqui (2018), a unit commitment constrained Nash equilibrium is given by $(q_1,u_1,\dots,q_I,u_I)$ such that i) $q_i$ is the best response of unit $i$ given $u_i$ and $q_{-i}$ for all $i=1,\dots,I$, i.e.,  
\begin{align*}
\sum_{t=1}^TP_{it}(q_{it},q_{-it})-\sum_{t=1}^TC_i(q_{it})-S_i(u_i)\geq \sum_{t=1}^TP_{it}(\bar q_{it},q_{-it})-\sum_{t=1}^TC_i(\bar q_{it})-S_i(u_i),\nonumber\\ 
\forall \bar q_i\in\{q_i\in\mathbb{R}^{T}:G_iq_i+H_iu_i\leq h_i\},i=1,\dots,I,
\end{align*}
and ii)
\begin{align*}
\sum_{t=1}^TP_{it}(q_{it},q_{-it})-\sum_{t=1}^TC_i(q_{it})-S_i(u_i)\geq \sum_{t=1}^TP_{it}(q_{it}^l,q_{-it})-\sum_{t=1}^TC_i(q_{it}^l)-S_i(u_i^l),\\
\forall u_i^l\in\{\{0,1\}^T:A_iu_i\leq b_i\}, i=1,\dots,I,
\end{align*}
where $q_i^l$ is the best response of unit $i$ given $u_i^l$
\begin{align*}
\sum_{t=1}^TP_{it}(q_{it}^l,q_{-it})-\sum_{t=1}^TC_i(q_{it}^l)-S_i(u_i^l)\geq \sum_{t=1}^TP_{it}(\bar q_{it},q_{-it})-\sum_{t=1}^TC_i(\bar q_{it})-S_i(u_i^l),\nonumber\\ 
\forall \bar q_i\in\{q_i\in\mathbb{R}^{T}:G_iq_i+H_iu_i^l\leq h_i\},i=1,\dots,I.
\end{align*}
Conditions ii) are equivalent to \eqref{eq17}-\eqref{eq13}, whereas condition i) is redundant with the constraints (\ref{eq27}) ensuring that $q_i=q_i^l$ if $u_i=u_i^l$. Thus, our definition of the unit commitment constrained Nash equilibrium is consistent with this.
\end{remark}

We proceed to study existence and uniqueness of the unit commitment constrained equilibrium. 

\begin{remark}
To analyse existence, consider the equilibrium problem that arises for fixed unit commitment of all units, which we refer to as the economic dispatch equilibrium. For fixed unit commitment $(u_1,\dots,u_I)$, assume that
\begin{enumerate}
\item[i)] the function $(q_i,q_{-i})\rightarrow\sum_{t=1}^TP_{it}(q_{it},q_{-it})-\sum_{t=1}^TC_i(q_{it})$ is continuous in $q_i$ and $q_{-i}=(q_1,\dots,q_{i-1},q_{i+1},\dots,q_T)$ and convex in $q_i$,
\item[ii)] the set $\{q_i\in\mathbb{R}^{T}:G_iq_i\leq h_i-H_iu_i\}$ is a non-empty and bounded polyhedron,
\end{enumerate}
for all $i=1,\dots,I$. Then, the economic dispatch equilibrium $(q_1,\dots,q_I)$ exists (Theorem 1.2 of \cite{Fudenberg1991}).

For any economic dispatch equilibrium $(q_1,\dots,q_I)$, we can find $(u_1,\dots,u_I)$ such that $q_i$ is the best response of unit $i$ given $u_i$ and $q_{-i}$ for all $i=1,\dots,I$. However, $(q_1,u_1,\dots,q_I,u_I)$ does not necessarily constitute an equilibrium, since we may not have that 
\begin{align*}
\sum_{t=1}^TP_{it}(q_{it},q_{-it})-\sum_{t=1}^TC_i(q_{it})-S_i(u_i)\geq \sum_{t=1}^TP_{it}(\bar q_{it},q_{-it})-\sum_{t=1}^TC_i(\bar q_{it})-S_i(\bar u_i),\\
\forall (\bar q_i,\bar u_i)\in\{(q_i,u_i)\in\mathbb{R}^{T}\times\{0,1\}^T:G_iq_i+H_iu_i\leq h_i, A_iu_i\leq b_i\}, i=1,\dots,I.
\end{align*}

\cite{Harwood2024} provide an example of a unit commitment constrained perfectly competitive Nash equilibrium, for which an equilibrium does not exist. Below, however, we provide sufficient conditions for an equilibrium to exist and show that these are satisfied for the unit commitment constrained Cournot-Nash equilibrium. 
\end{remark}

Consider the system-wide unit commitment problem: 
\begin{align}
\max_{q,u} \Big\{&\sum_{t=1}^TF_t(q_t)-\sum_{i=1}^I\sum_{t=1}^TC_i(q_{it})-\sum_{i=1}^IS_i(u_i) :\nonumber\\
&G_iq_i+H_iu_i\leq h_i, A_iu_i\leq b_i, \ u_i\in\{0,1\}^T,\ i=1,\dots,I\Big\},\label{eq14}
\end{align}
where $F_t:\mathbb{R}^I\rightarrow\mathbb{R}$ is a concave and continuously differentiable function. 

For fixed unit commitment and under mild assumptions, there exist functions $F_t(\cdot), t=1,\dots,T$ such that there is a one-to-one correspondence between economic dispatch equilibria and optimal solutions to the system-wide problem, see \cite{Egging2020}. For instance, under perfect and Cournot competition, these functions capture consumer surplus and market power adjusted consumer surplus, respectively. The following example demonstrates that this may not be the case for the unit commitment constrained equilibrium. If the system-wide problem maximises market power adjusted social welfare, there may be multiple equilibria, some of which may not be optimal solutions.

\begin{example}\label{ex3}
Consider the Cournot-Nash equilibrium problem
\begin{align}
\max_{q_i,u_i} \Big\{&(\beta-\alpha(q_1+q_2))q_i-c_iq_i-s_iu_i : 0\leq q_i\leq \overline q_iu_i, u_i\in\{0,1\}\Big\},i=1,2,\nonumber
\end{align}
and the system-wide maximisation of market power adjusted social welfare 
\begin{align}
\max_{q,u} \Big\{&\int_0^{q_1+q_2}(\beta-\alpha x)dx-\frac{1}{2}\alpha q_1^2-\frac{1}{2}\alpha q_1^2-c_1q_1-c_2q_2-s_1u_1-s_2u_2 : 0\leq q_i\leq \overline q_iu_i, u_i\in\{0,1\}, i=1,2\Big\}.\nonumber
\end{align}
For fixed unit commitment, an economic dispatch equilibrium is an optimal solution to the system-wide problem and vice versa. This does not hold in general for the unit commitment constrained equilibrium and system-wide optimisation. For instance, under Cournot competition and $\beta=5, \alpha=1, c_1=c_2=1, \overline{q}_1=\overline{q}_2=3, s_1=3, s_2=2$, $(q_1,u_1,q_2,u_2)=(2,1,0,0)$ is an equilibrium, but $(q_1,u_1,q_2,u_2)=(0,0,2,1)$ is the unique optimal solution to the system-wide problem. Thus, an equilibrium cannot necessarily be obtained by solving this problem. Nevertheless, the optimal solution $(q_1,u_1,q_2,u_2)=(0,0,2,1)$ is another equilibrium. 

\end{example}

In Example \ref{ex3}, an optimal solution to the system-wide problem is an equilibrium. This holds more generally under certain assumptions, i.e., when changes in the system-wide function $F_t(\cdot)$ with respect to the production of producer $i$ are the same as changes in the revenue $P_{it}(\cdot)$ function of this producer, for all $t=1,\dots,T$.

\begin{proposition}\label{prop3}
The following holds:
\begin{itemize}
\item[i)] If $F_t(q_{it},q_{-it})-F_t(x_i,q_{-it})=P_{it}(q_{it},q_{-it})-P_{it}(x_i,q_{-it})$ for all $q_t\in\mathbb{R}^{I}, x_i\in\mathbb{R}$ and $t=1,\dots,T, i=1,\dots,I$, then an optimal solution to the optimization problem (\ref{eq14}) is a solution to (\ref{eq15}), i.e., an equilibrium.
\item[ii)] Assume further that $P_{it}(q_{it},q_{-it})-P_{it}(\bar q_{it},q_{-it})=P_{it}(q_{it})-P_{it}(\bar q_{it}), \ t=1,\dots,T, i=1,\dots,I$. 
If $F_t(q_t)-F_t(x)=\sum_{i=1}^I(P_{it}(q_{it})-P_{it}(x_i))$ for all $q_t\in\mathbb{R}^{I}, x\in\mathbb{R}^I$ and $t=1,\dots,T$, then a solution to (\ref{eq15}), i.e., an equilibrium, is an optimal solution to the optimization problem (\ref{eq14}).
\end{itemize}
\end{proposition}

\emph{Proof:}
To show i), let $(q_1,u_1,\dots,q_I,u_I)$ be an optimal solution to (\ref{eq14}). Then, $G_iq_i+H_iu_i\leq h_i, \ A_iu_i\leq b_i, \ u_i\in\{0,1\}^T,\ i=1,\dots,I$ and 
\begin{align*}
\sum_{t=1}^TF_t(q_t)-\sum_{i=1}^I\sum_{t=1}^TC_i(q_{it})-\sum_{i=1}^IS_i(u_i)\geq \sum_{t=1}^TF_t(\bar q_t)-\sum_{i=1}^I\sum_{t=1}^TC_i(\bar q_{it})-\sum_{i=1}^IS_i(\bar u_i), \\ \forall (\bar q,\bar u)\in\{(q,u)\in\mathbb{R}^{TI}\times\{0,1\}^{TI}:G_iq_i+H_iu_i\leq h_i, \ A_iu_i\leq b_i,\ i=1,\dots,I\}.
\end{align*}
Let $i\in\{1,\dots,I\}$ and $\bar u_{-i}=u_{-i}$ and $\bar q_{-i}=q_{-i}$. Then, by the assumptions of Proposition \ref{prop3} i),
\begin{align*}
\sum_{t=1}^TP_{it}(q_{it},q_{-it})-\sum_{t=1}^TC_i(q_{it})-S_i(u_i)\geq \sum_{t=1}^TP_{it}(\bar q_{it},q_{-it})-\sum_{t=1}^TC_i(\bar q_{it})-S_i(\bar u_i), \\ \forall (\bar q_i,\bar u_i)\in\{(q_i,u_i)\in\mathbb{R}^{T}\times\{0,1\}^{T}:G_iq_i+H_iu_i\leq h_i, \ A_iu_i\leq b_i\}.
\end{align*}
Thus, $(q_1,u_1,\dots,q_I,u_I)$ is an equilibrium.

For ii), let $(q_1,u_1,\dots,q_I,u_I)$ be an equilibrium. Then, $G_iq_i+H_iu_i\leq h_i, \ A_iu_i\leq b_i, \ u_i\in\{0,1\}^T,\ i=1,\dots,I$ and
\begin{align*}
\sum_{t=1}^TP_{it}(q_{it})-\sum_{t=1}^TC_i(q_{it})-S_i(u_i)\geq \sum_{t=1}^TP_{it}(\bar q_{it})-\sum_{t=1}^TC_i(\bar q_{it})-S_i(\bar u_i), \\ \forall (\bar q_i,\bar u_i)\in\{(q_i,u_i)\in\mathbb{R}^{T}\times\{0,1\}^{T}:G_iq_i+H_iu_i\leq h_i, \ A_iu_i\leq b_i\}.
\end{align*}
for all $i=1,\dots,I$. Hence, by the assumptions of Proposition \ref{prop3} ii),
\begin{align*}
\sum_{t=1}^TF_t(q_t)-\sum_{i=1}^I\sum_{t=1}^TC_i(q_{it})-\sum_{i=1}^IS_i(u_i)\geq \sum_{t=1}^TF_t(\bar q_t)-\sum_{i=1}^I\sum_{t=1}^TC_i(\bar q_{it})-\sum_{i=1}^IS_i(\bar u_i), \\ \forall (\bar q,\bar u)\in\{(q,u)\in\mathbb{R}^{TI}\times\{0,1\}^{TI}:G_iq_i+H_iu_i\leq h_i, \ A_iu_i\leq b_i,\ i=1,\dots,I\}.
\end{align*}
As a result, $(q_1,u_1,\dots,q_I,u_I)$ is an optimal solution (\ref{eq14}).

\begin{remark}
The assumptions of Proposition \ref{prop3} i) are fulfilled under standard assumptions of an inverse affine demand curve, i.e., $P_{it}(q_t)=p_t(d_t)q_{it}$ with $p_t(d_t)=\beta_t-\alpha_t d_t, \alpha_t, \beta_t\geq 0$ and $d_t=\sum_{i=1}^Iq_{it}$, in Cournot-Nash equilibrium. Indeed, $P_{it}(q_{it},q_{-it})=(\beta_t-\alpha_t\sum_{k=1}^Iq_{kt})q_{it}$, and if we let
\begin{align*}
F_{t}(q_t)=\int_0^{d_t}p_t(x)dx-(1/2)\alpha_t \sum_{i=1}^Iq_{it}^2, 
\end{align*}
which is consumer surplus minus a market power adjustment term, then
\begin{align*}
F_t(q_{it},q_{-it})-F_t(x_i,q_{-it})=\beta_t(q_{it}-x_i)-(1/2)\alpha_t\Big(q_{it}^2-x_i^2+2(q_{it}-x_i)\sum_{k\neq i}^Iq_{kt}\Big)-(1/2)\alpha_t(q_{it}^2-x_i^2)\\
=\beta_t(q_{it}-x_i)-\alpha_t\Big(q_{it}\sum_{k=1}^Iq_{kt}-x_i\Big(\sum_{k\neq i}^Iq_{kt}+x_{it}\Big)\Big)=P_{it}(q_{it},q_{-it})-P_{it}(x_i,q_{-it}).
\end{align*}

Under perfect competition, it is not clear how to define $P_t(q_t)$ and $F_{t}(q_t)$. If we let
\begin{align*}
F_{t}(q_t)=\int_0^{d_t}p_t(x)dx,
\end{align*}
i.e., consumer surplus, an optimal solution to the system-wide problem is not necessarily an equilibrium, see e.g.\ the example of \cite{Harwood2024}. An exception is if demand is infinitely elastic, i.e., if $\alpha_t=0, \ t=1,\dots,T$.

Even in Cournot-Nash equilibrium, the assumptions of Proposition \ref{prop3} ii) are not necessarily satisfied under the standard assumptions. A necessary but strong assumption is that $F_{t}(q_t)=\sum_{i=1}^IP_{it}(q_{it})$ such that the system-wide problem is separable with respect to producers. In this case, it is obvious that there is a one-to-one correspondence between Nash equilibria and optimal solutions. As an example, this is the case if $\alpha_t=0,\ t=1,\dots,T$.
\end{remark}

Since an optimal solution to the system-wide problem is an equilibrium under the assumptions in Proposition \ref{prop3}, we can use this to establish existence of an equilibrium. We do this by ensuring that problem \eqref{eq14} is feasible and bounded. Boundedness is guaranteed by a bounded feasible region, e.g., if the constraints $G_iq_i+H_iu_i\leq h_i$ include $\underline q_iu_{it}\leq q_{it}\leq \overline q_iu_{it}, t=1,\dots,T$. 

\begin{corol}
Assume that $\{(q,u)\in\mathbb{R}^{TI}\times\{0,1\}^{TI}:G_iq_i+H_iu_i\leq h_i, \ A_iu_i\leq b_i,\ i=1,\dots,I\}$ is non-empty and bounded, and that the assumptions of Proposition \ref{prop3} i) hold. Then, there exists a unit commitment constrained Nash equilibrium.
\end{corol}

\emph{Proof:} By Wolsey and Nemhauser (Theorem 6.3), the problem
\begin{align*}
\max_{q,u} \Big\{&\sum_{t=1}^TF_t(q_t)-\sum_{i=1}^I\sum_{t=1}^TC_i(q_{it})-\sum_{i=1}^IS_i(u_i) :\nonumber\\
&(q,u)\in\text{conv}\{(q,u)\in\mathbb{R}^{TI}\times\{0,1\}^{TI}:G_iq_i+H_iu_i\leq h_i, \ A_iu_i\leq b_i,\ i=1,\dots,I\}\Big\}.
\end{align*}
has an optimal solution that is also optimal for problem \eqref{eq14}, and thus, is an equilibrium. 

\begin{remark}
Note that it is possible to define other optimisation problems than (\ref{eq14}) for which an optimal solution is an equilibrium, i.e., a solution to (\ref{eq15}). Consider, for instance, the problem 
\begin{align*}
\max_{q,u} \Big\{&\sum_{t=1}^TF_t(q_t)-\sum_{i=1}^I\sum_{t=1}^TC_i(q_{it})-V(u) :\nonumber\\
&G_iq_i+H_iu_i\leq h_i, A_iu_i\leq b_i, \ u_i\in\{0,1\}^T,\ i=1,\dots,I\Big\},
\end{align*}
where $V:\mathbb{R}^{IT}\rightarrow\mathbb{R}$. By the proof of Proposition \ref{prop3} i), a sufficient condition is that $V(u)-V(x_i,u_{-i})=S_i(u_i)-S_i(x_i)$ for all $u_i,x_i\in\mathbb{R}^T$.
\end{remark}

\setcounter{example}{5}
\begin{example}
Consider the Cournot-Nash equilibrium problem of Example \ref{ex3} and the system-wide problem
\begin{align}
\max_{q,u} \Big\{&\int_0^{q_1+q_2}(\beta-\alpha x)dx-\frac{1}{2}\alpha q_1^2-\frac{1}{2}\alpha q_1^2-c_1q_1-c_2q_2-s_1s_2u_1u_2 : 0\leq q_i\leq \overline q_iu_i, u_i\in\{0,1\}, i=1,2\Big\}.\nonumber
\end{align}
or equivalently, the mixed-integer convex quadratic program
\begin{align*}
\max_{q,u} \Big\{&\int_0^{q_1+q_2}(\beta-\alpha x)dx-\frac{1}{2}\alpha q_1^2-\frac{1}{2}\alpha q_1^2-c_1q_1-c_2q_2-s_1s_2v: 0\leq q_i\leq \overline q_iu_i,\\
&u_1+u_2-1\leq v, v\leq u_1, v\leq u_2, \ u_i\in\{0,1\}, i=1,2\Big\}.\nonumber
\end{align*}
The optimal solutions $(q_1,u_1,q_2,u_2)=(2,1,0,0)$ and $(q_1,u_1,q_2,u_2)=(0,0,2,1)$ are also equilibria. 
\end{example}

The fact that an optimal solution to the system-wide problem is not necessarily an equilibrium (cf.\ the example of Harwood et al. (2024)) justifies the characterisation (\ref{eq12})-(\ref{eq27}). Even if an optimal solution is an equilibrium, there can be other equilibria. The characterisation can likewise be used to determine multiple equilibria. To find an alternative equilibrium $(q_1,u_1,\dots,q_I,u_I)$, we add to (\ref{eq12})-(\ref{eq27}) the constraint
\begin{align}
\sum_{(i,t):u_{it}^j=1}(1-u_{it})+\sum_{(i,t):u_{it}^j=0}u_{it}\geq 1,\label{eq28}
\end{align}
ensuring that $u\neq u^j$. If the objective of the economic dispatch problems is strictly concave for fixed $(u_1,\dots,u_I)$ (e.g., under Cournot competition with $\alpha_t>0, t=1,\dots,T$) such that the optimal $(q_1,\dots,q_I)$ is unique, then all equilibria can be determined. For the procedure, see Algorithm \ref{alg1}.

\RestyleAlgo{plain}
\begin{algorithm}\label{alg1}
\caption{Multiple equilibria.}
Let $(q_1,u_1,\dots,q_I,u_I)$ be a solution to the system (\ref{eq12})-(\ref{eq27}) and $j\gets 1$.\\
\While{A solution exists}{
\vspace{2mm} $(q_1^j,u_1^j,\dots,q_I^j,u_I^j) \gets (q_1,u_1,\dots,q_I,u_I)$\\
Add the constraint (\ref{eq28}) to current system:
\begin{align*}
&(\ref{eq12})-(\ref{eq27}),\\
&\sum_{(i,t):u_{it}^l=1}(1-u_{it})+\sum_{(i,t):u_{it}^l=0}u_{it}\geq 1, \ l=1,\dots,j-1
\end{align*}
and let $(q_1,u_1,\dots,q_I,u_I)$ be a solution if such exists\\
$j\gets j+1$\vspace{2mm} 
}
\end{algorithm}

Note that even if there is no guarantee that an optimal solution to the system-wide problem is an equilibrium, it is easy to check if this solution satisfies (\ref{eq12})-(\ref{eq27}) in the initialisation of Algorithm \ref{alg1}.

\section{Numerical experiments}

We carry out numerical experiments for a case study of the all-island Irish electricity market, with a dataset including both technical and economic features of 52 generating units spread across 16 producers. 

\subsection{Case study}

We consider a unit commitment constrained Nash equilibrium between producers $k=1,\dots,K$. In the following, we assume Cournot competition, but it is easy to adjust the formulation to perfect competition. Each producer $k$ operates a set of units $i\in I_k$, with $u_{kt}=(u_{it}: i\in I_k)$ and $q_{kt}=(q_{it}: i\in I_k)\in\mathbb{R}^{|I_k|}$ representing commitments and production levels, respectively. Variable operating costs are given by the convex quadratic and non-decreasing function $C_i(q)=a_iq+b_iq^2$ with $a_i,b_i\geq 0$ and fixed online and start-up costs are $s_i^{ON}\geq 0$ and $s_i^{SU}\geq 0$, respectively. Inverse market demand is given by the affine and non-increasing function $p_t(d_t)=\beta_t-\alpha_td_t$ with $\beta_t,\alpha_t\geq 0$ and market clearing ensures that total supply meets demand, i.e., $d_t=\sum_{k=1}^KQ_{kt}, Q_{kt}=\sum_{i\in I_k}q_{it}$ for $t=1,\dots,T$. 

The equilibrium problem is:
\begin{align*}
\max_{q_k,u_k}\Big\{&\sum_{t=1}^TP_{kt}(q_{kt},q_{-kt})-\sum_{t=1}^TC_k(q_{kt})-S_k(u_k): \\
&\underline{q}_iu_{it}\leq q_{it}\leq \overline{q}_iu_{it}, u_{it}\in\{0,1\}, t=1,\dots,T, i\in I_k
\Big\}, k=1,\dots,K,
\end{align*}
where
\begin{align*}
P_{kt}(q_{kt},q_{-kt})=p_t(d_t)\sum_{i\in I_k}q_{it}, d_t=\sum_{k=1}^KQ_{kt}, Q_{kt}=\sum_{i\in I_k}q_{it}, t=1,\dots,T, k=1,\dots,K\\
C_k(q_{kt})=\sum_{i\in I_k}(a_iq_{it}+b_iq_{it}^2), t=1,\dots,T, k=1,\dots,K\\[-2mm]
S_k(u_k)=\sum_{t=1}^T\sum_{i\in I_k}(s_i^{ON}u_{it}+s_i^{SU}\max\{u_{it}-u_{it-1},0\})), k=1,\dots,K.
\end{align*}
This is equivalent to the system
\begin{align*}
\underline{q}_iu_{it}\leq q_{it}\leq \overline{q}_iu_{it}, u_{it}\in\{0,1\}, t=1,\dots,T, i\in I_k, k=1,\dots,K\\
\sum_{t=1}^TP_{kt}(q_{kt},q_{-kt})-\sum_{t=1}^TC_k(q_{kt})-S_k(u_k)\geq \sum_{t=1}^TP_{kt}(q_{kt}^j,q_{-kt})-\sum_{t=1}^TC_k(q_{kt}^j)-S_k(u_k^j),\\ 
j=1,\dots,J_k,k=1,\dots,K,\\
-p_t(Q_{kt}^l,Q_{-kt})-\alpha_t+a_i+2b_iq_{it}^l+\overline{\lambda}_{it}^l-\underline{\lambda}_{it}^l=0, \\ 
0\leq \overline{\lambda}_{it}^l\perp \overline{q}_iu_i^l-q_{it}^l\geq 0, \ 0\leq \underline{\lambda}_{it}^l\perp q_{it}^l-\underline{q}_iu_i^l\geq 0, i\in I_k, l=1,\dots,L_{kt}, t=1,\dots,T,k=1,\dots,K,\\[2mm]
Q_{kt}^l=\sum_{i\in I_k}q_{it}^l, l=1,\dots,L_{kt}, t=1,\dots,T,k=1,\dots,K,\\
Q_{kt}^l-\underline{M}_{kl}\Big(\sum_{i\in I_k:u_{it}^l=1}(1-u_{it})+\sum_{i\in I_k:u_{it}^l=0}u_{it}\Big)\leq Q_{kt}\leq Q_{kt}^l+\overline{M}_{kl}\Big(\sum_{i\in I_k:u_{it}^l=1}(1-u_{it})+\sum_{i\in I_k:u_{it}^l=0}u_{it}\Big), \\ 
l=1,\dots,L_{kt}, t=1,\dots,T,k=1,\dots,K\nonumber
\end{align*}
with $J_k=2^{|I_k|T}, L_{kt}=2^{|I_k|}$ and $q_{it}^j=q_{it}^l, i\in I_k$ if $u_{it}^j=u_{it}^l, i\in I_k$. The number of complementarity constraints increases exponentially with $|I_k|$ but only linearly with $K$ and $T$. 

The system-wide unit commitment problem is:
\begin{align*}
\max_{q,u} \Big\{&\sum_{t=1}^TF_t(q_t)-\sum_{k=1}^K\sum_{t=1}^TC_k(q_{kt})-\sum_{k=1}^KS_k(u_k) :\nonumber\\
&\underline{q}_iu_{it}\leq q_{it}\leq \overline{q}_iu_{it}, u_{it}\in\{0,1\}, t=1,\dots,T, i\in I_k, k=1,\dots,K\Big\}
\end{align*}
with 
\begin{align*}
F_{t}(q_t)=\int_0^{d_t}p_t(x)dx-(1/2)\alpha_t \sum_{k=1}^KQ_{kt}^2, \ d_t=\sum_{k=1}^KQ_{kt}, \ Q_{kt}=\sum_{i\in I_k}q_{it}.
\end{align*}
By Proposition \ref{prop3} i), the optimal solution is a unit commitment constrained Nash-Cournot equilibrium (to see this, replace $q_i$ and $u_i$ by $q_k=(q_{it}:i\in I_k)$ and $u_k=(u_{it}:i\in I_k)$, respectively, in the proof). 

\subsection{Data}

We use the data set of \cite{Devine2025} representing the all-island Irish electricity market. Accordingly, $K=16$, $I=52$, $\alpha_t:=\alpha=0.137$ Euro/MWh$^2$ and $b_i:=b=0.000213$ Euro/MWh$^2$. The rest of the data can be found online, see \cite{Devine2025data}. Moreover, we take the number of time periods as $T=24$ hours.

Our results include both perfect and Cournot competition. We justify Cournot competition as follows. The Irish power grid being an isolated system makes it vulnerable to market concentration and exercise of market power. In fact, in the data set, one dominant firm owns more than 35\% of total capacity, and two firms each own more than 10\%. For a discussion of market design and market power issues in the all-island Irish electricity market, see also \cite{Cosmo2016}. For simplicity and like \cite{Devine2025}, we assume that all producers are Cournot players. 

\subsection{Unit commitment constrained market clearing}

To compare the market equilibria with and without discreteness and demonstrate the importance of accounting for unit commitment decisions, we consider Cournot competition and solve the mixed-integer (convex) quadratic program (MIQP) of system-wide unit commitment and the (convex) quadratic program (QP) of economic dispatch. Whereas the former produces a unit commitment constrained Nash equilibrium, the latter provides a Nash equilibrium in which all producers solve an economic dispatch problem. The unit commitment constrained Nash equilibrium is guaranteed by the assumption that all producers are Cournot players. Nevertheless, this assumption can be relaxed, allowing only larger producers to make unit commitment decisions and exercise market power while the remaining producers are perfectly competitive and solve convex economic dispatch problems. As a sensitivity analysis with respect to the variations in demand, we consider both real data and randomly generated data for the intercept of the inverse demand curve. The random demand data are uniformly distributed over time between the minimum and maximum values of the real data. As a result, the random data fluctuate significantly more over time than the real data, as reflected in the autocorrelations of -0.10 and 0.96, respectively. 

The MIQP and QP were implemented in GAMS and solved by the mixed-integer quadratically constrained programming (MIQCP) routine and the quadratically constrained programming (QCP) routine in GUROBI, respectively, using a 3.1GHz i9-9900 processor with 32GB of RAM. For real demand, computation times were 3347 and 461 seconds for MIQP and QP, respectively, demonstrating the feasibility of obtaining a unit commitment constrained Cournot-Nash equilibrium by optimisation, but also illustrating the significant increase in complexity due to discreteness. The impact of discreteness, however, is more moderate for random demand, with computation times of 948 and 456 seconds for MIQP and QP, respectively.

The effects of accounting for unit commitment decisions are already studied in detail by \cite{Devine2025}, including the comparison of perfect and Cournot competition and the sensitivity with respect to the share of renewables in the system. Here, we illustrate the influence of unit commitment on market prices and volumes and provide summary statistics for producer and consumer surplus and social welfare. Market prices and volumes are depicted in Figs.\ \ref{fig1} and \ref{fig2} for the real demand data. Prices are generally higher with discrete unit commitment than without due to the internalisation of start-up and fixed costs. The variation in prices is likewise higher with discreteness. In this case, the mean and standard deviation are 14\% and 24\% higher, respectively. In contrast, the mean and standard deviation of total market volumes are 5\% and 7\% lower, respectively, reflecting the inertia caused by start-up costs. Columns 2-4 of Table \ref{table1} shows consumer surplus of the MIQP and the QP, total producer surplus (profit) and social welfare of the MIQP. For the QP solution, producer surplus and social welfare include the costs of the unit commitment  corresponding to production. That is, for an optimal solution $q=(q_{it})$ to the QP, we let $u_{it}=1$ if $\underline{q}_i\leq q_{it}\leq \overline{q}_i$ and $u_{it}=0$ if $q_{it}=0$ for $i\in I_k, k=1,\dots,K, t=1,\dots,T$, and add the unit commitment costs to those of economic dispatch. We find that failure to account for unit commitment results in an overestimation of consumer surplus by 2\% (computed as the difference in consumer surplus of the MIQP and the QP as a share of the former), a total profit loss of 58\% and a total welfare loss of 10\%. Similar results for the random demand data can be found in Columns 5-7 of Table \ref{table1} and in Figs.\ \ref{fig3} and \ref{fig4} of Appendix \ref{app:randomdemand}. The underestimation of prices and profits, the overestimation of volumes and consumer surplus, and the significant loss of social welfare, have important implications for market design and clearly illustrate the relevance a unit commitment constrained approach to market clearing.

\begin{figure}[!h]
\centering
\begin{tikzpicture}
\pgfplotsset{
    scale only axis,
    xmin=1, xmax=24
}
\begin{axis}[
    width=10cm,
    height=4cm,
    xlabel=hours,
    ymin=0, ymax=350,
    ylabel=price (Euro/MWh),
    legend style={at={(1.15,1)},anchor=north west, draw=none}
]
\addplot[] coordinates {
(1,188.6960152) (2,185.0645623) (3,167.389359) (4,164.4038902) 
(5,148.2414387) (6,131.0877015) (7,145.8503382) (8,149.1673094) 
(9,127.9857457) (10,166.8029469) (11,179.325299) (12,183.7495574) 
(13,182.2858649) (14,197.9520942) (15,195.9843641) (16,188.7814812) 
(17,230.79354) (18,323.5773585) (19,295.0110684) (20,281.4416965) 
(21,244.3635139) (22,216.094918) (23,186.5592963) (24,189.1233559)
};
\addlegendentry{MIQP}
\addplot[dotted, thick] coordinates {
(1,169.9059233) (2,157.9055839) (3,147.0464167) (4,136.7810815)
(5,131.795748) (6,120.7517845) (7,126.1998681) (8,126.7703702)
(9,118.8904735) (10,131.8197121) (11,144.9828255) (12,156.441107)
(13,167.3390633) (14,177.0828978) (15,174.3166931) (16,169.9401481)
(17,207.2734457) (18,248.7229464) (19,237.7720727) (20,231.0021812)
(21,214.2680614) (22,188.6139527) (23,169.050305) (24,170.0770469)
};
\addlegendentry{QP}
\end{axis}
\end{tikzpicture}\caption{Market prices, real data.}\label{fig1}
\end{figure}
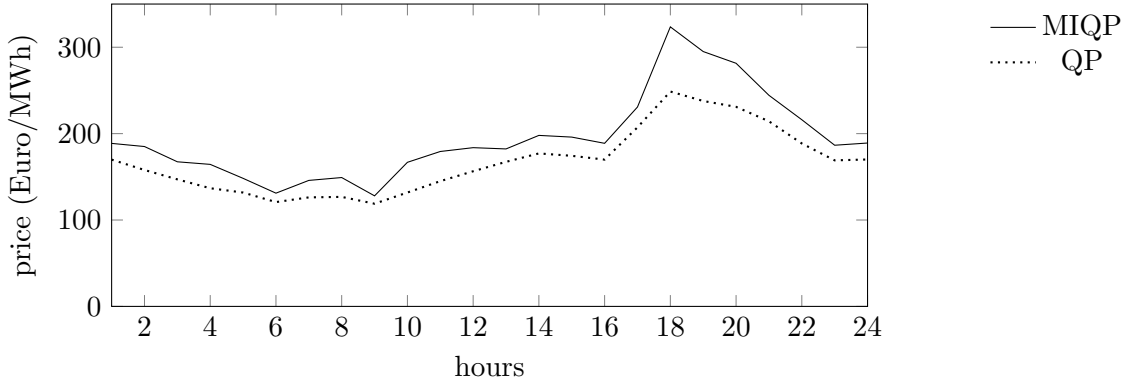

\begin{figure}[!h]
\centering
\begin{tikzpicture}
\pgfplotsset{
    scale only axis,
    xmin=1, xmax=24
}
\begin{axis}[
    width=10cm,
    height=4cm,
    xlabel=hours,
    ymin=0, ymax=7000,
    ylabel=volume (MWh),
    legend style={at={(1.15,1)},anchor=north west, draw=none},
]
\addplot[] coordinates {
(1,4570.331623) (2,4135.805919) (3,3700.698378) (4,3188.891321)
(5,3075.053127) (6,2702.524706) (7,2808.801222) (8,2811.838533)
(9,2657.259941) (10,2940.415083) (11,3492.883829) (12,4074.200547)
(13,4523.509093) (14,4741.942121) (15,4675.568958) (16,4570.955942)
(17,5100.800999) (18,5615.800999) (19,5504.800999) (20,5388.800999)
(21,5154.800999) (22,4802.800998) (23,4554.724128) (24,4573.453144)
};
\addlegendentry{MIQP}
\addplot[dotted, thick] coordinates {
(1,4707.796569) (2,4334.49614) (3,3849.5237) (4,3390.974842)
(5,3195.366854) (6,2778.140424) (7,2952.560541) (8,2975.690529)
(9,2723.799324) (10,3196.346163) (11,3744.127214) (12,4273.984279)
(13,4632.857216) (14,4894.617424) (15,4834.085759) (16,4708.795761)
(17,5272.869796) (18,6163.422471) (19,5923.551239) (20,5757.807459)
(21,5374.973939) (22,5003.846819) (23,4682.816773) (24,4712.792529)
};
\addlegendentry{QP}
\end{axis}
\end{tikzpicture}\caption{Market volumes, real data.}\label{fig2}
\end{figure}
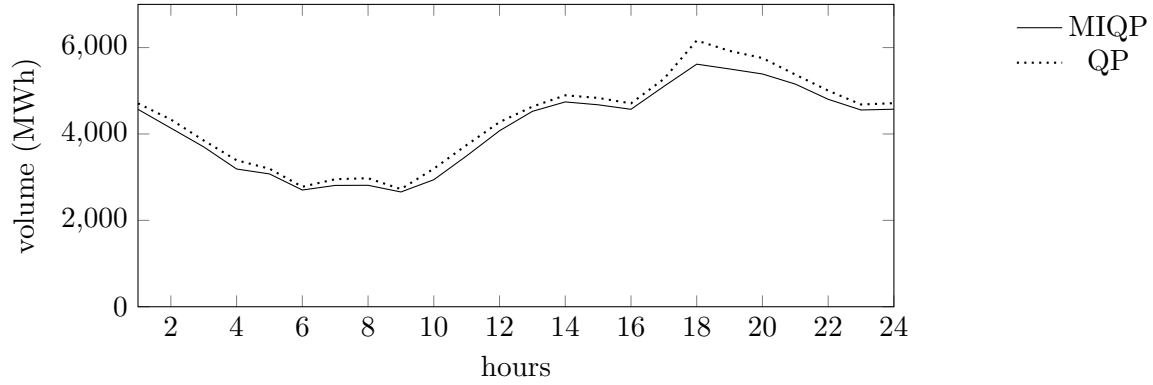

\begin{table}\caption{Consumer surplus (CS), total producer surplus (PS), and social welfare (SW) for the real and random demand data and the MIQP and QP solutions. Note that for the QP, producer surplus and social welfare includes the costs of unit commitment. All values are in 1,000 Euro.\\}\label{table1}
\begin{center}
\begin{tabular}{ lcccccc } 
\hline
&\multicolumn{3}{c}{real}&\multicolumn{3}{c}{random}\\
\hline
& CS & PS & SW& CS & PS & SW\\ 
 \hline
MIQP& 49,887 & 11,202 & 44,314 &  59,389 & 14,037 & 53,327\\ 
QP& 50,831 & 4,717 & 39,820 & 60,628 & 7,282 &49,120\\ 
 \hline
\end{tabular}
\end{center}
\end{table}

\subsection{Multiplicity of equilibria}

We proceed to check for multiple equilibria under both perfect and Cournot competition. Due to the computational complexity of Algorithm \ref{alg1}, we consider only the four larger units of the dataset (Whitegate, Ardnacrusha1, AghadaCCGT and GreatIslandCCGT, where Ardnacrusha1 in fact represents the combination of several small hydro units) and assume each unit acts as an independent producer the market. We likewise reduce the number of time periods to $T=8$ hours.

Fig. \ref{fig5} reflects the unit commitment schedule, showing total online capacity in the market for three distinct equilibria under Cournot competition. Furthermore, Fig. \ref{fig6} and \ref{fig7} depict total production volumes and prices. For similar results under perfect competition, see Figs.\ \ref{fig8}-\ref{fig10} of Appendix \ref{app:multeq}. Clearly, unit commitment may vary substantially across equilibria (the differences between maximum and minimum are in the range 2475-4667 MW under Cournot competition and 375-2475 MW under perfect competition), although total production varies less (the differences are in the range 112-1867 MW under Cournot competition and 194-1154 MW under perfect competition). The large differences in online capacities are reflected in prices, which likewise vary significantly across equilibria (the differences are in the range 15-256 Euro/MW under Cournot competition and 27-206 Euro/MWh under perfect competition). 

Table \ref{table2} shows consumer and producer surplus and social welfare of the three equilibria. Consistent with capacities, production and prices, these vary considerably across equilibria. Under perfect competition, maximum consumer surplus is 74\% higher than the minimum, whereas the numbers are 22\% and 24\% for producer surplus and social welfare, respectively. Under Cournot competition, these numbers are 39\%, 36\% and 24\%. Evidently, the choice of equilibrium has crucial impact on the system metrics. 

It should be noted that, for the unit commitment constrained equilibrium, the ability of producers to affect market prices does not necessarily result in higher prices, higher profits and lower social welfare than under perfect competition. Indeed, Table \ref{table2} include examples of lower total profits and higher social welfare under Cournot, cf.\ Equil.\ 1 and Equil.\ 2 under Cournot and perfect competition, respectively (although for any of the equilibria obtained in the numerical experiments, at least one producer always has higher profits under Cournot than under perfect competition). Thus, unit commitment constrained equilibria and their implications for market power should be interpreted with caution.


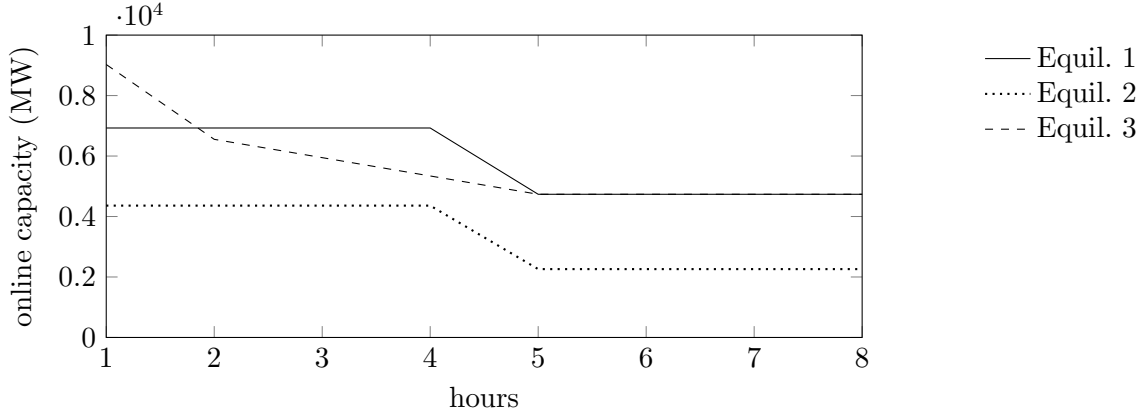
\begin{figure}[!h]
\centering
\begin{tikzpicture}
\pgfplotsset{
    scale only axis,
    xmin=1, xmax=8
}
\begin{axis}[
    width=10cm,
    height=4cm,
    xlabel=hours,
    ymin=0, ymax=10000,
    ylabel=online capacity (MW),
    legend style={at={(1.15,1)},anchor=north west, draw=none},
]
\addplot[] coordinates {
(1,6927.66) (4,6927.66) (5,4735.36) (8,4735.36)
};
\addlegendentry{Equil. 1}
\addplot[dotted, thick] coordinates {
(1,4360.24) (4,4360.24) (5,2260.50) (8,2260.50)
};
\addlegendentry{Equil. 2}
\addplot[dashed] coordinates {
(1,9027.4) (2,6552.54) (5,4735.36) (8,4735.36)
};
\addlegendentry{Equil. 3}
\end{axis}
\end{tikzpicture}\caption{Online capacity, Cournot competition.}\label{fig5}
\end{figure}


\begin{figure}[!h]
\centering
\begin{tikzpicture}
\pgfplotsset{
    scale only axis,
    xmin=1, xmax=8
}
\begin{axis}[
    width=10cm,
    height=4cm,
    xlabel=hours,
    ymin=0, ymax=7000,
    ylabel=volume (MWh),
    legend style={at={(1.15,1)},anchor=north west, draw=none},
]
\addplot[] coordinates {
(1,4835.38)	(2,3481.14)	(3,3230.36)	(4,3292.17)	(5,2075.07)	(6,2260.5)	(7,2260.5)	(8,2975.96)
};
\addlegendentry{Equil. 1}
\addplot[dotted, thick] coordinates {
(1,4006.39)	(2,3776.39)	(3,3051.6)	(4,2850.44)	(5,2186.98)	(6,2130.81)	(7,2127.92)	(8,1947.15)
};
\addlegendentry{Equil. 2}
\addplot[dashed] coordinates {
(1,2968.65)	(2,2738.65)	(3,2457.08)	(4,2190.84)	(5,2075.07)	(6,1826.64)	(7,2758.66)	(8,1947.15)
};
\addlegendentry{Equil. 3}
\end{axis}
\end{tikzpicture}\caption{Market volumes (production), Cournot competition.}\label{fig6}
\end{figure}
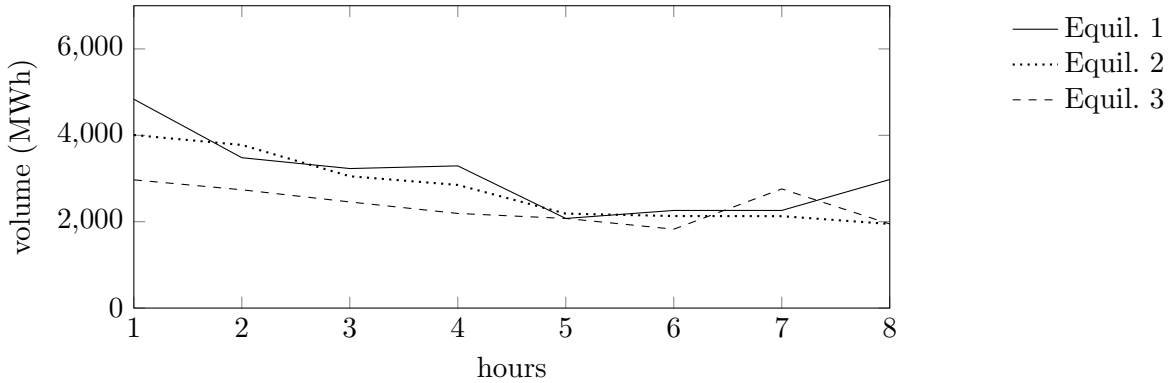


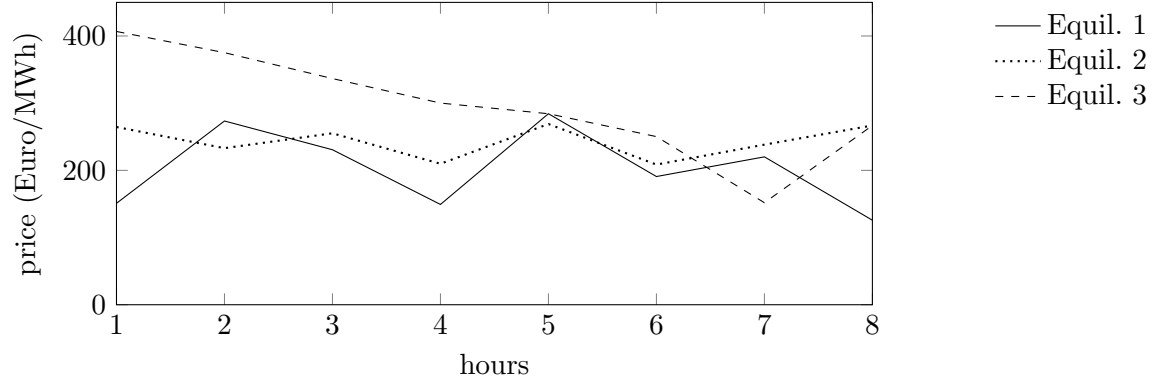
\begin{figure}[!h]
\centering
\begin{tikzpicture}
\pgfplotsset{
    scale only axis,
    xmin=1, xmax=8
}
\begin{axis}[
    width=10cm,
    height=4cm,
    xlabel=hours,
    ymin=0, ymax=450,
    ylabel=price (Euro/MWh),
    legend style={at={(1.15,1)},anchor=north west, draw=none},
]
\addplot[] coordinates {
(1,150.96)
(2,273.47)
(3,230.68)
(4,149.26)
(5,284.29)
(6,190.81)
(7,220.1)
(8,125.81)
};
\addlegendentry{Equil. 1}
\addplot[dotted, thick] coordinates {
(1,264.53)
(2,233.02)
(3,255.17)
(4,209.78)
(5,268.95)
(6,208.58)
(7,238.27)
(8,266.76)
};
\addlegendentry{Equil. 2}
\addplot[dashed] coordinates {
(1,406.7)
(2,375.19)
(3,336.62)
(4,300.14)
(5,284.29)
(6,250.25)
(7,151.85)
(8,266.76)
};
\addlegendentry{Equil. 3}
\end{axis}
\end{tikzpicture}\caption{Market prices, Cournot competition.}\label{fig7}
\end{figure}


\begin{table}\caption{Consumer surplus (CS), total producer surplus (PS), and social welfare (SW) for two perfectly competitive equilibria. All values are in 1,000 Euro.\\}\label{table2}
\begin{center}
\begin{tabular}{ lcccccc } 
\hline
&\multicolumn{3}{c}{Cournot}&\multicolumn{3}{c}{perfect}\\
\hline
& CS & PS & SW& CS & PS & SW\\ 
 \hline
Equil. 1&  5,491 & 3,233 & 8,723& 5,152 & 3,895 & 9,047\\ 
Equil. 2& 4,479 & 3,155 &7,634& 3,701 & 4,160 & 7,860 \\ 
Equil. 3& 3,164 & 3,846 &7,010& 4,258 & 3,066 & 7,324\\ 
 \hline
\end{tabular}
\end{center}
\end{table}

To illustrate the complexity of the equation system (\ref{eq12})-(\ref{eq27}) and Algorithm \ref{alg1}, we report computation time for each iteration (i.e., time to obtain a new equilibrium) as a function of the number of time periods $T$. Figure \ref{fig11} indicates that computation times increase exponentially with the number of time periods, consistent with the exponential increase in the number of constraints. For $T=8$, computation time was more than 15 minutes, and for $T=9$, it was 1 hour and 37 minutes. For the higher number of time periods, computation time can be reduced by initially checking whether the system-wide optimal solution is an equilibrium (which is it under both perfect and Cournot competition in this case study), cf.\ the first iteration.

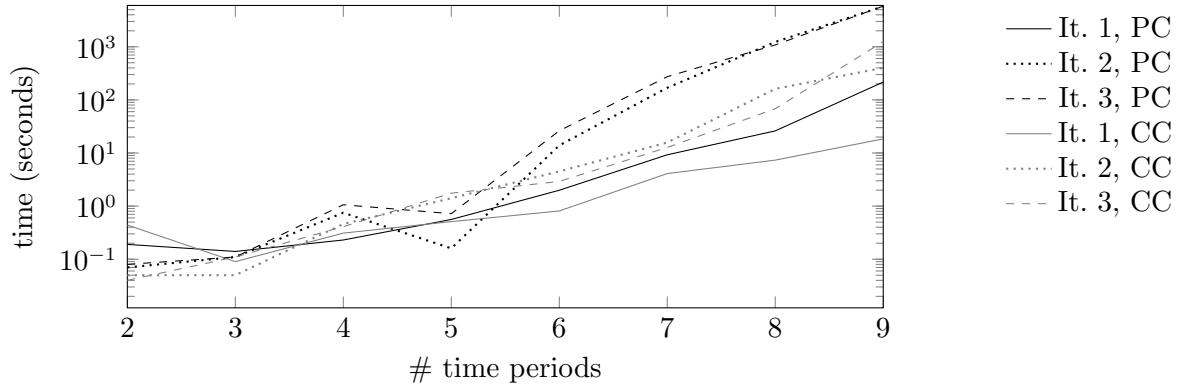
\begin{figure}[!h]
\centering
\begin{tikzpicture}
\pgfplotsset{
    scale only axis,
    xmin=2, xmax=9
}
\begin{axis}[
    width=10cm,
    height=4cm,
    xlabel=\# time periods,
    ymin=0, ymax=6000, ymode=log,
    ylabel=time (seconds),
    legend style={at={(1.15,1)},anchor=north west, draw=none},
]
\addplot[] coordinates {
(2,0.19)
(3,0.14)
(4,0.23)
(5,0.56)
(6,1.99)
(7,9.25)
(8,26.08)
(9,216.82)
};
\addlegendentry{It.\ 1, PC}
\addplot[dotted, thick] coordinates {
(2,0.07)
(3,0.11)
(4,0.75)
(5,0.16)
(6,13.89)
(7,169.14)
(8,1233.44)
(9,5762.78)
};
\addlegendentry{It.\ 2, PC}
\addplot[dashed] coordinates {
(2,0.08)
(3,0.11)
(4,1.06)
(5,0.72)
(6,26.2)
(7,273.56)
(8,1085.51)
(9,5822.39)
};
\addlegendentry{It.\ 3, PC}
\addplot[gray] coordinates {
(2,0.44)
(3,0.09)
(4,0.31)
(5,0.51)
(6,0.81)
(7,4.1)
(8,7.34)
(9,18.44)
};
\addlegendentry{It.\ 1, CC}
\addplot[dotted, thick, gray] coordinates {
(2,0.05)
(3,0.05)
(4,0.46)
(5,1.41)
(6,4.51)
(7,15.75)
(8,161.24)
(9,402.79)
};
\addlegendentry{It.\ 2, CC}
\addplot[dashed, gray] coordinates {
(2,0.04)
(3,0.11)
(4,0.41)
(5,1.75)
(6,2.91)
(7,12.71)
(8,68.24)
(9,1247.06)
};
\addlegendentry{It.\ 3, CC}
\end{axis}
\end{tikzpicture}\caption{Computation time as a function of time, under perfect competition (PC) and Cournot competition (CC).}\label{fig11}
\end{figure}

\section{Conclusion}

This paper proposes a formulation of unit commitment constrained Nash equilibrium in which self-scheduling profit-maximizing producers solve mixed-integer convex programming problems subject to joint market clearing constraints. We derive optimality conditions and valid inequalities to arrive at a system of equations that characterises a Nash equilibrium and develop a procedure for obtaining multiple equilibria. Under Cournot competition, an equilibrium can likewise be obtained by a single mixed-integer convex programming problem. A case study illustrates the impact of unit commitment on market prices and volumes and the social welfare losses incurred by ignoring unit commitment, but also the variation across equilibria. The system-wide optimization offers market operators and policy-makers practical decision support, using standard optimisation software and without the need for specialised algorithms. For market design in particular, our work suggests that unit commitment can be efficiently incorporated into market clearing, though caution is needed when multiple equilibria exist in the market.

Our unit commitment problem does not account for network constraints. However, it would be straightforward to include such joint constraints in the market clearing without affecting the results of the paper. Furthermore, although our approach is framed in the context of a unit commitment constrained Nash equilibrium, it can be generalized to mixed-binary convex Nash games. The unit commitment constrained Nash-Cournot equilibrium is one such example for which the equilibrium can be obtained by mixed-integer convex programming. When a unit commitment constrained Nash equilibrium does not exist, it is possible to use our system of inequalities (\ref{eq12})-(\ref{eq27}), introduce slack variables in (\ref{eq17}) and minimize the disequilibrium \`a la \cite{Harwood2024}.  

\section{Acknowledgements}
Trine Krogh Boomsma was supported by the Independent Research Fund Denmark, project number 0217-00009B and the Carlsberg Foundation, grant CF24/0897. Mel Devine acknowledges funding from Research Ireland and co-funding partners under grant number 21/SPP/3756 through the NexSys Strategic Partnership Programme.

\vspace{5mm}
\bibliographystyle{informs2014} 
\bibliography{ref.bib} 

\clearpage
\newpage

\section{Appendix}
\vspace{5mm}

\subsection{Proof of Proposition \ref{prop5}:}\label{app:prop5} 
Our proof is constructive in the sense that it not only verifies the validity of the inequalities but it also demonstrates how tight they are. Let $t\in\{1,\dots,T\},i\in\{1,\dots,I\}, j\in\{1,\dots,J\}$ with $u_{it}^j=1$. Consider the inequalities
\begin{align}
&q_{it}^j-\underline{M}_{ij}\Big(\sum_{(k,s):u_{ks}^j=1}(1-u_{ks})+\sum_{(k,s):u_{k,s}^j=0}u_{ks}\Big)-M_1u_{it}-M_2(1-u_{it})\leq q_{it}\nonumber\\
&\leq q_{it}^j+\underline{M}_{ij}\Big(\sum_{(k,s):u_{ks}^j=1}(1-u_{ks})+\sum_{(k,s):u_{k,s}^j=0}u_{ks}\Big)+M_3u_{it}+M_4(1-u_{it}).\label{eq11}
\end{align}
with constants $\underline{M}_{ij}, \overline{M}_{ij}, M_1,M_2,M_3,M_4$. If $u_{it}=1, u=u^j$, (\ref{eq11}) reduces to
\begin{align*}
&-M_1\leq 0\leq M_3,
\end{align*}
i.e., $M_1,M_3\geq 0$ for the inequalities to be valid. Assume that $\underline{M}_{ij}, \overline{M}_{ij}\geq 0$. If $u_{it}=1, u\neq u^j$, then it is sufficient that 
\begin{align*}
&q_{it}^j-\underline{M}_{ij}-M_1\leq q_{it}\leq q_{it}^j+\overline{M}_{ij}+M_3.
\end{align*}
If $u_{it}=0, u\neq u^j$, then it suffices that 
\begin{align*}
&q_{it}^j-\underline{M}_{ij}-M_2\leq 0\leq q_{it}^j+\overline{M}_{ij}+M_4.
\end{align*}
Thus, if we take $M_1=M_3=0, M_2=\underline{q}_i, M_4=-\overline{q}_i, \underline{M}_{ij}=\overline{M}_{ij}=\overline{q}_i-\underline{q}_i$, the inequalities (\ref{eq11}) are valid.

\subsection{Market prices and volumes for the random demand data}\label{app:randomdemand}
\vspace{0.2cm}

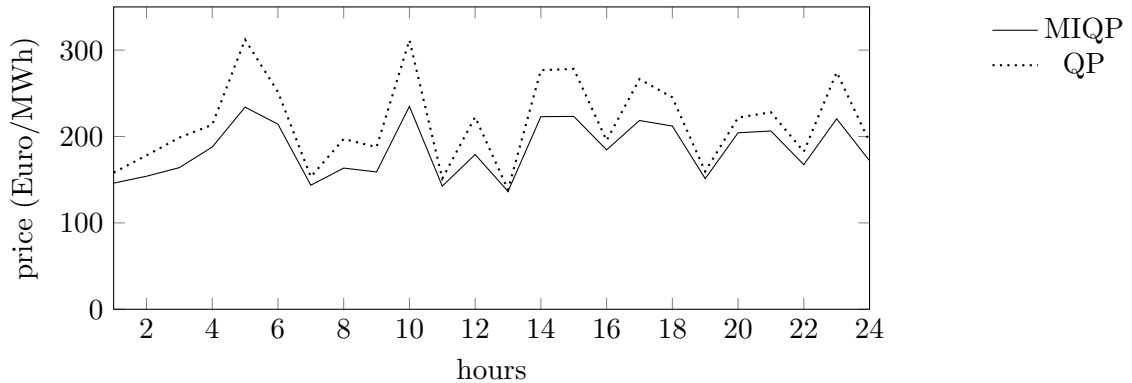
\begin{figure}[!h]
\centering
\begin{tikzpicture}
\pgfplotsset{
    scale only axis,
    xmin=1, xmax=24
}
\begin{axis}[
    width=10cm,
    height=4cm,
    xlabel=hours,
    ymin=0, ymax=350,
    ylabel=price (Euro/MWh),
    legend style={at={(1.15,1)},anchor=north west, draw=none},
]
\addplot[] coordinates {
(1,146.0337642) (2,154.0080212) (3,163.9551126) (4,187.7668663) (5,234.0118043) 
(6,214.205967) (7,143.6486312) (8,163.4269225) (9,158.9838868) (10,234.8672379) 
(11,142.5508969) (12,179.1625227) (13,136.4193204) (14,222.9423156) (15,223.2847047) 
(16,184.5624307) (17,218.5254844) (18,212.1067898) (19,151.248555) (20,204.3033292) 
(21,206.3568716) (22,167.3960359) (23,220.5113476) (24,171.9426378)
};
\addlegendentry{MIQP}
\addplot[dotted, thick] coordinates {
(1,158.2079032) (2,178.119057) (3,198.7069277) (4,213.5590106)
(5,312.1124822) (6,251.5587864) (7,153.2784886) (8,197.1337682)
(9,187.7632087) (10,311.5657222) (11,151.3268183) (12,222.5804972)
(13,138.4191173) (14,276.8464441) (15,278.2133447) (16,195.3554105)
(17,266.4579995) (18,245.2710401) (19,159.5143887) (20,221.8970399)
(21,228.0480926) (22,182.428144) (23,274.386023) (24,193.782251)
};
\addlegendentry{QP}
\end{axis}
\end{tikzpicture}\caption{Market prices, random data.}\label{fig3}
\end{figure}

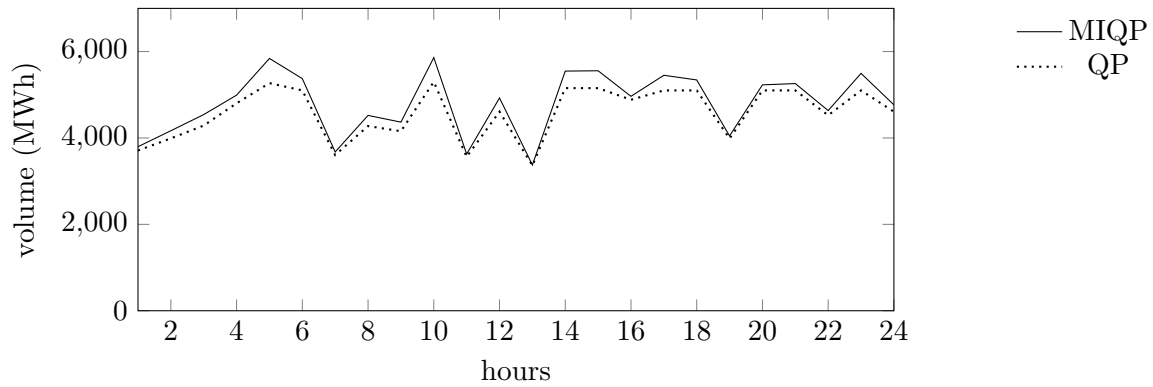
\begin{figure}[!h]
\centering
\begin{tikzpicture}
\pgfplotsset{
    scale only axis,
    xmin=1, xmax=24
}
\begin{axis}[
    width=10cm,
    height=4cm,
    xlabel=hours,
    ymin=0, ymax=7000,
    ylabel=volume (MWh),
    legend style={at={(1.15,1)},anchor=north west, draw=none},
]
\addplot[] coordinates {
(1,3797.803137) (2,4167.4649) (3,4541.693761) (4,4991.491236) (5,5841.171552) 
(6,5374.067511) (7,3676.252342) (8,4522.557906) (9,4366.062359) (10,5859.913353) 
(11,3627.28317) (12,4926.439073) (13,3375.140685) (14,5548.153949) (15,5555.649092) 
(16,4963.934313) (17,5451.466695) (18,5343.424715) (19,4046.652661) (20,5229.513435) 
(21,5259.490087) (22,4634.520582) (23,5494.938475) (24,4767.258458)
};
\addlegendentry{MIQP}
\addplot[dotted, thick] coordinates {
(1,3708.739316) (2,3991.072881) (3,4287.455706) (4,4802.800528) (5,5269.800981) 
(6,5100.800974) (7,3605.80203) (8,4275.96466) (9,4155.517994) (10,5298.800979) 
(11,3563.0801) (12,4608.800948) (13,3360.51053) (14,5153.801) (15,5153.801) 
(16,4884.974806) (17,5100.801) (18,5100.800999) (19,3986.181302) (20,5100.800999) 
(21,5100.800999) (22,4524.548372) (23,5100.801) (24,4607.483759)
};
\addlegendentry{QP}
\end{axis}
\end{tikzpicture}\caption{Market volumes, random data.}\label{fig4}
\end{figure}

\newpage
\subsection{Capacity, production and prices under perfect competition}\label{app:multeq}
\vspace{0.2cm}

\begin{figure}[!h]
\centering
\begin{tikzpicture}
\pgfplotsset{
    scale only axis,
    xmin=1, xmax=8
}
\begin{axis}[
    width=10cm,
    height=4cm,
    xlabel=hours,
    ymin=0, ymax=10000,
    ylabel=online capacity (MW),
    legend style={at={(1.15,1)},anchor=north west, draw=none},
]
\addplot[] coordinates {
(1,4735.36) (8,4735.36)
};
\addlegendentry{Equil. 1}
\addplot[dotted, thick] coordinates {
(1,4360.24) (3,4360.24) (4,2260.50) (8,2260.50)
};
\addlegendentry{Equil. 2}
\addplot[dashed] coordinates {
(1,4360.24) (3,4360.24) (4,2260.50) (6,2260.50) (7,2099.74) (8,2260.50)
};
\addlegendentry{Equil. 3}
\end{axis}
\end{tikzpicture}\caption{Online capacity, perfect competition.}\label{fig8}
\end{figure}
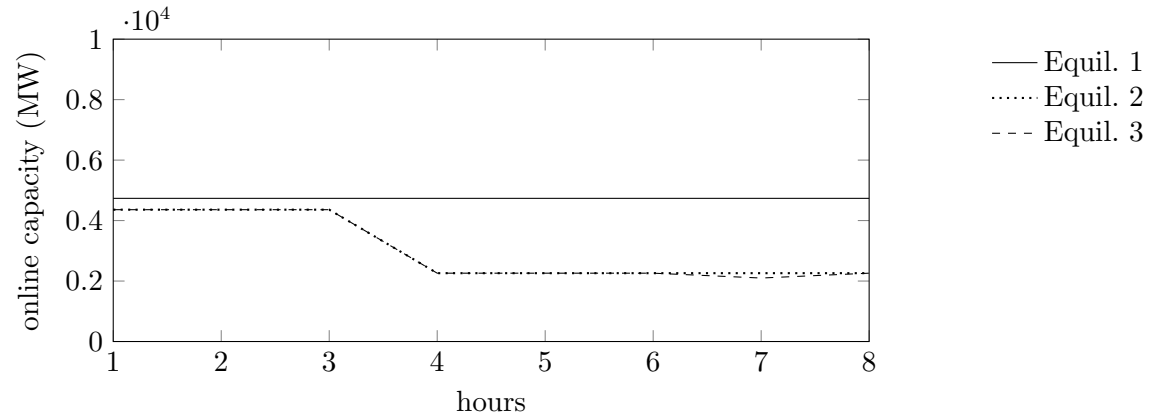


\begin{figure}[!h]
\centering
\begin{tikzpicture}
\pgfplotsset{
    scale only axis,
    xmin=1, xmax=8
}
\begin{axis}[
    width=10cm,
    height=4cm,
    xlabel=hours,
    ymin=0, ymax=7000,
    ylabel=volume (MWh),
    legend style={at={(1.15,1)},anchor=north west, draw=none},
]
\addplot[] coordinates {
(1,4591.42)	(2,2716.4)	(3,3568.29)	(4,3762.61)	(5,2610.75)	(6,2307.41)	(7,1106.19)	(8,2548.43)
};
\addlegendentry{Equil. 1}
\addplot[dotted, thick] coordinates {
(1,3458.59)	(2,2998.59)	(3,2856.81)	(4,2260.5)	(5,2260.5)	(6,2113.9)	(7,2260.5)	(8,2260.5)
};
\addlegendentry{Equil. 2}
\addplot[dashed] coordinates {
(1,4306.77)	(2,2998.59)	(3,3568.29)	(4,2260.5)	(5,1671.44)	(6,2113.9)	(7,2099.74)	(8,2260.5)
};
\addlegendentry{Equil. 3}
\end{axis}
\end{tikzpicture}\caption{Market volumes (production), perfect competition.}\label{fig9}
\end{figure}
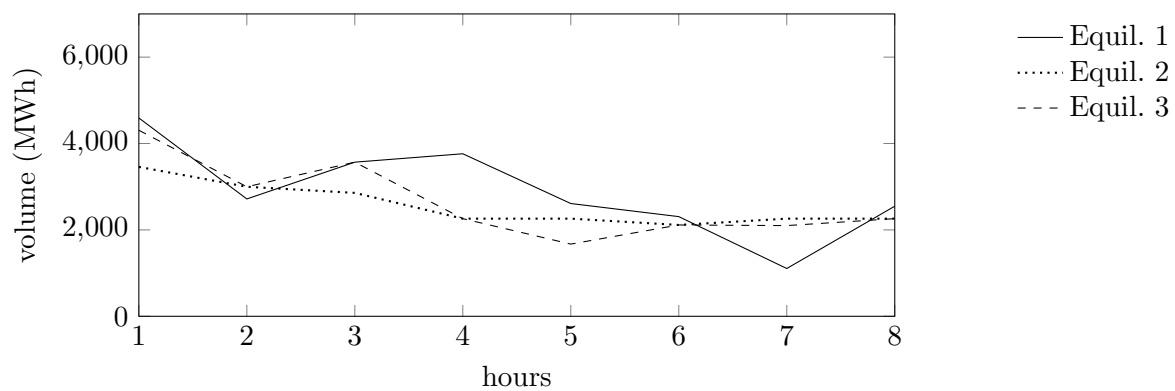


\begin{figure}[!h]
\centering
\begin{tikzpicture}
\pgfplotsset{
    scale only axis,
    xmin=1, xmax=8
}
\begin{axis}[
    width=10cm,
    height=4cm,
    xlabel=hours,
    ymin=0, ymax=450,
    ylabel=price (Euro/MWh),
    legend style={at={(1.15,1)},anchor=north west, draw=none},
]
\addplot[] coordinates {
(1,184.38)
(2,378.24)
(3,184.38)
(4,84.81)
(5,210.9)
(6,184.38)
(7,378.24)
(8,184.38) 
};
\addlegendentry{Equil. 1}
\addplot[dotted, thick] coordinates {
(1,339.58)
(2,339.58)
(3,281.86)
(4,290.6)
(5,258.88)
(6,210.9)
(7,220.1)
(8,223.83)
};
\addlegendentry{Equil. 2}
\addplot[dashed] coordinates {
(1,223.38)
(2,339.58)
(3,184.38)
(4,290.6)
(5,339.58)
(6,210.9)
(7,242.13)
(8,223.83)
};
\addlegendentry{Equil. 3}
\end{axis}
\end{tikzpicture}\caption{Market prices, perfect competition.}\label{fig10}
\end{figure}
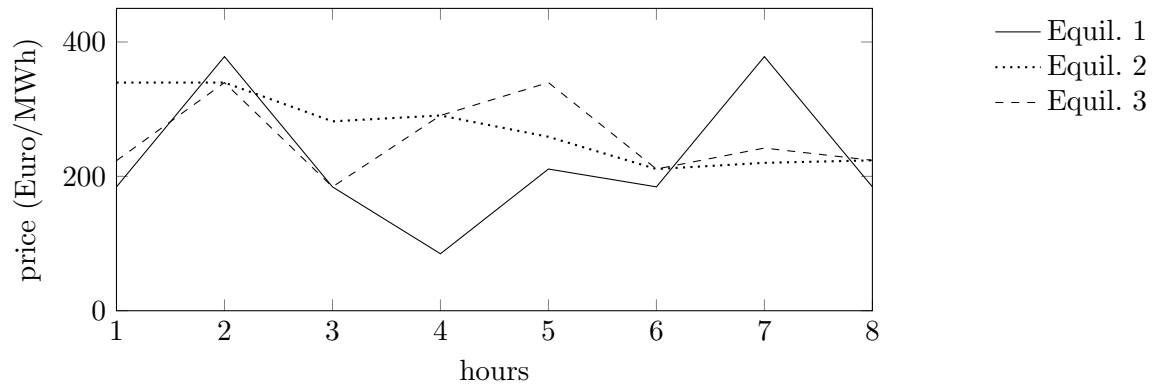


\end{document}